\documentclass[]{aspm}

\articleinfo{}{}{}

\usepackage{verbatim}
\usepackage{amssymb}
\usepackage{amsbsy}
\usepackage{amscd}
\usepackage{amsmath}
\usepackage{amsthm}
\usepackage{stmaryrd}
\usepackage[mathscr]{eucal}
\usepackage[all]{xy}

\newtheorem{theorem}{Theorem}[section]
\newtheorem*{theoremA*}{Theorem A}
\newtheorem*{theoremB*}{Theorem B}
\newtheorem*{theoremC*}{Theorem C}
\newtheorem*{theoremD*}{Theorem D}
\newtheorem{lemma}[theorem]{Lemma}
\newtheorem{proposition}[theorem]{Proposition}
\newtheorem{corollary}[theorem]{Corollary}
\newtheorem{claim}[theorem]{Claim}
\newtheorem{conjecture}[theorem]{Conjecture}

\theoremstyle{definition}
\newtheorem{definition}[theorem]{Definition}

\theoremstyle{remark}
\newtheorem{remark}[theorem]{Remark}
\newtheorem{notation}[theorem]{Notation}
\newtheorem{question}[theorem]{Question}

\numberwithin{theorem}{section}
\numberwithin{equation}{theorem}

\DeclareFontFamily{OMS}{rsfs}{\skewchar\font'60}
\DeclareFontShape{OMS}{rsfs}{m}{n}{<-5>rsfs5 <5-7>rsfs7 <7->rsfs10 }{}
\DeclareSymbolFont{rsfs}{OMS}{rsfs}{m}{n}
\DeclareSymbolFontAlphabet{\scr}{rsfs}

\newcommand{\autoref}{\ref}

\newcommand{\mustata}{Musta{\c{t}}{\u{a}}}
\newcommand{\Z}{\mathbb {Z}}
\DeclareMathOperator{\Spec}{{Spec}}
\newcommand{\bQ}{\mathbb{Q}}
\newcommand{\Q}{\mathbb {Q}}
\newcommand{\ba}{\mathfrak{a}}
\newcommand{\R}{\mathbb {R}}
\newcommand{\bC}{\mathbb{C}}
\DeclareMathOperator{\an}{{an}}
\newcommand{\mydot}{{{\,\begin{picture}(1,1)(-1,-2)\circle*{2}\end{picture}\ }}}
\newcommand{\bH}{\mathbb{H}}
\newcommand{\myR}{{\bf R}}
\newcommand{\myH}{{\bf h}}

\newcommand{\sH}{\scr{H}}
\newcommand{\mJ}{\mathcal{J}}
\newcommand{\Ram}{\mathrm{Ram}}
\DeclareMathOperator{\Tr}{Tr}
\DeclareMathOperator{\sHom}{{\sH}om}
\newcommand{\tensor}{\otimes}
\newcommand{\sL}{\scr{L}}
\newcommand{\bZ}{\mathbb{Z}}
\newcommand{\DuBois}[1]{{\underline \Omega {}^0_{#1}}}
\newcommand{\tld}{\widetilde }
\DeclareMathOperator{\red}{red}
\newcommand{\cf}{{\itshape cf.} }
\newcommand{\bF}{\mathbb{F}}
\newcommand{\Sch}{\mathrm{Sch}}
\newcommand{\Shv}{\mathrm{Shv}}
\newcommand{\Nis}{\mathrm{Nis}}

\newcommand{\mR}{\mathrm{R}}
\newcommand{\perf}{\mathrm{perf}}
\newcommand{\colim}{\mathop{\mathrm{colim}}}
\newcommand{\calO}{{\mathcal{O}}}
\newcommand{\cO}{{\mathcal{O}}}
\renewcommand{\myR}{\mR}
\newcommand{\infinity}{\infty}
\newcommand{\ie}{{\itshape ie} }
\DeclareMathOperator{\Image}{Im}
\newcommand{\sN}{\scr{N}}
\DeclareMathOperator{\Div}{{div}}
\newcommand{\blank}{\underline{\hskip 10pt}}
\renewcommand{\O}{\cO}

\title[The weak ordinarity conjecture and $F$-singularities]{The weak ordinarity conjecture and $F$-singularities}

\author[B. Bhatt]{Bhargav Bhatt}

\author[K. Schwede]{Karl Schwede}

\author[S. Takagi]{Shunsuke Takagi}

\address{Department of Mathematics, University of Michigan, Ann Arbor, MI 48109, USA}

\address{Department of Mathematics, University of Utah, 155 South 1400 East, Salt Lake City, UT 84112, USA}

\address{Graduate School of Mathematical Sciences, University of Tokyo 3-8-1 Komaba Meguro-ku, Tokyo 153-8914, Japan}

\email{bhargav.bhatt@gmail.com}

\email{schwede@math.utah.edu}

\email{stakagi@ms.u-tokyo.ac.jp}

\rcvdate{}
\rvsdate{}



\begin{document}

\begin{abstract}

Recently M.~\mustata{} and V.~Srinivas related a natural conjecture about the Frobenius action on the cohomology of the structure sheaf after reduction to characteristic $p > 0$ with another conjecture connecting multiplier ideals and test ideals.  We generalize this relation to the case of singular ambient varieties.  Additionally, we connect these results to a conjecture relating $F$-injective and Du Bois singularities.  Finally, using an unpublished result of Gabber, we also show that $F$-injective and Du Bois singularities have a common definition in terms of smooth hypercovers.

\end{abstract}

\maketitle

\section{Introduction}

In \cite{MustataSrinivasOrdinary,MustataOrdinary2} \mustata{}  and Srinivas related a very natural open problem on ordinarity of smooth projective varieties after reduction to characteristic $p > 0$, with another open question on the connection between multiplier ideals and test ideals after reduction to characteristic $p > 0$.

We need some notation to state these conjectures.
Let $X$ be a scheme of finite type over a field $k$ of characteristic zero. Choosing a suitable finitely generated $\Z$-subalgebra $A$ of $k$, we can construct a scheme $X_A$ of finite type over $A$ such that $X_A \otimes_A k \cong X$.
We refer $X_A$ as a model of $X$ over $A$ and denote by $X_s$ the fiber of $X_A$ over a closed point $s \in \Spec A$.
See the paragraph just before Proposition \ref{prop.TestIdealsCoincidesWithMultiplierIdeals} for more details on reducing from characteristic zero to positive characteristic.
The first of the conjectures is then stated as follows.
\begin{conjecture}\label{MS conj} \textnormal{(The weak ordinarity conjecture \cite[Conjecture 1.1]{MustataSrinivasOrdinary})}
Let $V$ be an $n$-dimensional smooth projective variety over a field $k$ of characteristic zero.
Given a model of $V$ over a finitely generated $\Z$-subalgebra $A$ of $k$, there exists a Zariski-dense set of closed points $S  \subseteq \Spec A$ such that the action induced by Frobenius on $H^n(V_s, \mathcal{O}_{V_s})$ is bijective for every $s \in S$.
\end{conjecture}

Since the ordinarity of  $V_s$ in the sense of \cite{BlochKatoPAdicEtale} implies that the action of Frobenius on $H^n(V_s, \mathcal{O}_{V_s})$ is bijective (see \cite[Remark 5.1]{MustataSrinivasOrdinary}) and the converse does not hold in general (see \cite[Example 5.7]{JoshiRajanOrdinarity}), we call Conjecture \ref{MS conj} as the weak ordinarity conjecture.

\begin{remark}
\label{rem.FinitelyManySmoothIsOk}
As pointed out in \cite[Remark 5.2]{MustataSrinivasOrdinary}, if Conjecture \ref{MS conj} holds, then we can also find a dense set of closed points $S \subseteq \Spec A$ that satisfies the condition of the conjecture for any set of finitely many varieties $X^{(i)}$ over $k$.

In order to prove Conjecture \ref{MS conj} it is sufficient to assume that $k$ is algebraically closed.  Indeed, the reduction to characteristic $p > 0$ process does not care differentiate between $X$ and $X \times_k {\overline{k}}$.
Hence we assume that $k$ is algebraically closed as necessary.
\end{remark}

We now state a variant of the conjectured relationship between test ideals and multiplier ideals under reduction to characteristic $p > 0$.  The difference between this conjecture and the one presented in \cite{MustataSrinivasOrdinary,MustataOrdinary2} is that $X$ is not smooth and we also include a boundary $\bQ$-divisor $\Delta$.

\begin{conjecture}\label{test ideal conj}\textnormal{(A generalization to singular $X$ of \cite[Conjecture 1.2]{MustataSrinivasOrdinary})}
Let $X$ be a normal variety over a field $k$ of characteristic zero.
Suppose that $\Delta$ is a $\Q$-divisor on $X$ such that $K_X+\Delta$ is $\Q$-Cartier, and $\ba$ is a nonzero ideal on $X$.
Given a model of $(X, \Delta, \ba)$ over a finitely generated $\Z$-subalgebra $A$ of $k$, there exists a Zariski-dense set of closed points $S  \subseteq \Spec A$ such that
\begin{equation}\label{test ideal=mult ideal}
\tau(X_s, \Delta_s, \ba_s^{\lambda})=\mathcal{J}(X, \Delta, \ba^{\lambda})_s \textup{ }
\end{equation}
for all $\lambda \in \R_{\ge 0}$ and all $s \in S$.
Furthermore, if we have finitely many triples $(X_i, \Delta_i, \ba_i)$ as above and corresponding models over $A$, then there is a dense subset of closed points in $\Spec A$ such that $(\ref{test ideal=mult ideal})$ holds for each of these triples.
\end{conjecture}

Our first main result generalizes the main result of \cite{MustataSrinivasOrdinary}.

\begin{theoremA*}{\textnormal{(Theorem \ref{thm.EquivalentConjectures})}}
Conjecture \ref{test ideal conj} holds if and only if Conjecture \ref{MS conj} holds.
\end{theoremA*}

It is worth noting that we do not employ semi-stable reduction, the use of which was the key obstruction to generalizing the results of \cite{MustataSrinivasOrdinary} to the case of singular varieties.

We also explore the relation between $F$-injective and Du Bois singularities from two different perspectives.  Du Bois singularities are Hodge theoretic in origin, and were coined by Steenbrink \cite{SteenBrinkDuBoisReview} based upon the work of Deligne and Du~Bois \cite{DeligneHodgeIII,DuBoisMain}.  Roughly speaking, Du Bois singularities are the local condition on a proper variety over $\bC$ that naturally guarantees that the map $H^i(X^{\an}, \mathbb{C}) \to H^i(X^{\an}, \cO_{X^{\an}})$ is surjective for all $i > 0$ (see \cite{KovacsTheIntuitiveDefinitionOfDuBois} for additional discussion).  On the other hand, a variety $X$ in characteristic $p > 0$ is called $F$-injective if for every point $x \in X$, the Frobenius map on the local cohomology module $H^i_{x}(\cO_{X,x}) \to H^i_{x}(\cO_{X,x})$ is injective for each $i$.
It is a result of the second author \cite{SchwedeFInjectiveAreDuBois} that if $X$ has dense $F$-injective type\footnote{meaning its reduction to positive characteristic is $F$-injective for a Zariski-dense set of closed points $S \subseteq \Spec A$. See Definition \autoref{def.DuBoisSingularities} for the detail.} then $X$ has Du~Bois singularities.  We point out that once Conjecture \autoref{MS conj} is assumed, it easily follows from the methods of \cite{MustataSrinivasOrdinary} that the converse holds as well.  In fact, we even have the following:

\begin{theoremB*} {\textnormal{(Theorem \autoref{thm.DuBoisImpliesFInjectiveType})}}
Conjecture \autoref{MS conj} holds if and only if for every Du Bois scheme of finite type over a field of characteristic zero has dense $F$-injective type.
\end{theoremB*}

We also give a philosophical explanation for why $F$-injective and Du Bois singularities are related: they have a common definition.

\begin{theoremC*} {\textnormal{(Theorem \autoref{thm.DuBoisFinjectiveUnified})}}
Suppose that $X$ is a reduced scheme of finite type over a field $k$.  Then for any proper hypercover with smooth terms $\pi_{\mydot} : X_{\mydot} \to X$, consider the canonical map
\begin{equation}
\label{eq.LocalCohomologyDuBoisThmC}
H^i_{z}(\cO_X) \to \bH^i_z(\myR (\pi_{\mydot})_* \cO_{X_{\mydot}})
\end{equation}
for each $i$ and each closed point $z \in X$.
\begin{itemize}
\item{} Suppose $k$ is of characteristic zero, then $X$ has Du Bois singularities if and only if \eqref{eq.LocalCohomologyDuBoisThmC} injects for each $\pi_{\mydot}$, $i$ and $z$.
\item{} Suppose $k$ is of characteristic $p > 0$ and $F$-finite, then $X$ has $F$-injective singularities if and only if  \eqref{eq.LocalCohomologyDuBoisThmC} injects for each $\pi_{\mydot}$, $i$ and $z$.
\end{itemize}
\end{theoremC*}
In characteristic zero, this is not new.  It was observed by Kov\'acs in \cite{KovacsDuBoisLC1} (in fact, the map \eqref{eq.LocalCohomologyDuBoisThmC} is always surjective by \cite{KovacsSchwedeDBDeforms}). The proof of Theorem C relies on an unpublished result of O.~Gabber. To explain this result, note first that the direct limit of $\myR(\pi_\mydot)_* \cO_{X_\mydot}$ over all proper hypercovers of $X$ computes the (derived) pushforward of the sheafification of $\cO$ for Voevodsky's $h$-topology down to the Zariski topology.  Work of Deligne and Du Bois (see also \cite{HuberJorder,Lee}), which lies at the heart of Kov\'acs' observation mentioned above, identifies this pushforward in characteristic zero: it can be computed as the pushforward of the structure sheaf from {\em any} fixed proper hypercover with smooth terms of $X$.  Gabber's result simply gives the characteristic $p$ counterpart of this identification:

\begin{theoremD*}[O.~Gabber, Theorem \autoref{GabberTheorem}]
The $h$-sheafification of the structure sheaf in characteristic $p$ coincides with its perfection, and this identification extends to cohomology. In particular, for any excellent characteristic $p$ scheme $X$ and any $i \geq 0$, one has a natural isomorphism
\begin{small}
\[
\lim_{\rightarrow} H^i(X_\mydot, \cO_{X_{\mydot}}) \simeq \lim_{\rightarrow} \Big(H^i(X,\cO_X) \stackrel{F^*}{\to} H^i(X,\cO_X) \stackrel{F^*}{\to} H^i(X,\cO_X) \to \dots \Big).
\]
\end{small}
Here the limit runs over all proper hypercovers $\pi_{\mydot} : X_{\mydot} \to X$, and $F^*$ denotes the pullback along the Frobenius endomorphism of $X$.
\end{theoremD*}

\vskip 12pt
\noindent
\begin{small}
{\it Acknowledgements:}
The authors thank O.~Gabber for explaining his above result to us and allowing us to include it in this work.
The authors would like to thank Manuel Blickle, S\'andor Kov\'acs, Shrawan Kumar, Vikram Mehta, Mircea \mustata{} and Kevin Tucker for useful discussions.  We would like to thank Christohper Hacon and Linquan Ma for pointing out a mistake in our original proof of Lemma 3.9. We also thank the referee for many useful comments.
The first author was partially supported by the NSF grants DMS \#1128155 and DMS \#1340424.
The second author was partially supported by the NSF grant DMS \#1064485, NSF FRG Grant DMS \#1265261/1501115, NSF CAREER Grant DMS \#1252860/1501102  and by a Sloan Fellowship.
The third author was partially supported by Grant-in-Aid for Young Scientists (B) 23740024 and for Scientific Research (C) 26400039 from JSPS.
The material is based upon work supported by the National Science Foundation under Grant No. 0932078 000, initiated while the authors were in residence at the Mathematical Science Research Institute (MSRI) in Berkeley, California, during the spring semester 2013.
\end{small}
\section{Preliminaries}

Throughout this paper, all schemes are Noetherian and all morphisms of schemes are separated.
In this language, a variety is a reduced irreducible algebraic scheme of finite type over a field.

We first recall definitions of multiplier ideals and modules.
Our main reference for this is \cite{LazarsfeldPositivity2}.

\begin{definition}[Multiplier ideals and modules]
Suppose that $X$ is a normal variety over a field of characteristic zero.
Further, suppose that $\Delta$ is an effective $\bQ$-divisor, $\ba$ is an ideal sheaf and $t \geq 0$ is a real number.
Let $\pi : Y \to X$ be a proper birational morphism with $Y$ normal such that $\ba \cdot \cO_Y = \cO_Y(-G)$ is invertible, and we assume that $K_X$ and $K_Y$ agree wherever $\pi$ is an isomorphism.
\begin{itemize}
\item[(a)]  If $K_X + \Delta$ is $\bQ$-Cartier, we assume that $\pi$ is a log resolution of $(X, K_X + \Delta, \ba)$.  Then we define the \emph{multiplier ideal} to be
\[
\hspace*{2em} \mJ(X, \Delta, \ba^t) = \pi_* \cO_Y(\lceil K_Y - \pi^*(K_X + \Delta) - t G\rceil) \subseteq \cO_X.
\]
\item[(b)]  If $\Delta$ is $\bQ$-Cartier, we assume that $\pi$ is a log resolution of $(X, \Delta, \ba)$.  Then we define the \emph{multiplier module} to be
\[
\mJ(\omega_X, \Delta, \ba^t) = \pi_* \cO_Y(\lceil K_Y - \pi^*\Delta  - t G\rceil) \subseteq \omega_X.
\]
\end{itemize}
Both of these objects are independent of the choice of a resolution.
If $\Delta$ is not effective, one can still define the multiplier ideal and module by the same formulas as above, and they are then fractional ideals.
\end{definition}

\begin{remark}
In characteristic $p > 0$ and even in mixed characteristic, we will frequently be in the situation where we also have a fixed log resolution $\pi : Y \to X$ of all the relevant terms.  In such a situation, we will use $\mJ(X, \Delta, \ba^t)$ and $\mJ(\omega_X, \Delta, \ba^t)$ to denote the multiplier ideal and modules with respect to that fixed log resolution.
\end{remark}

\begin{lemma}
\label{lem.PropertiesOfMultiplier}
With notation as above:
\begin{itemize}
\item[(i)]  $\mJ(X, \Delta, \ba^t) = \mJ(\omega_X, K_X + \Delta, \ba^t)$.
\item[(ii)]  If $D$ is a Cartier divisor, then
\[
\mJ(X, \Delta +D, \ba^t) = \mJ(X, \Delta, \ba^t) \otimes \cO_X(-D).
\]
\item[(iii)]  Suppose that $f : Y \to X$ is a finite surjective morphism with $Y$ normal.
Let $\Ram_{f}$ denote the ramification divisor of $f$ and $\Tr_{f} : f_* K(Y) \to K(X)$ the trace map.  Then
\[
\hspace*{2em}
\begin{array}{rl}
\mJ(X, \Delta, \ba^t)\hspace*{-0.7em}&=\Tr_{f}\big(f_* \mJ(Y, \pi^* \Delta - \Ram_f, (\ba \cO_Y)^t)\big)\\
&=\big(f_* \mJ(Y, \pi^* \Delta - \Ram_f, (\ba \cO_Y)^t)\big) \cap K(X),\\
\mJ(\omega_X, \Delta, \ba^t)\hspace*{-0.7em}&=\Tr_{f}\big(f_* \mJ(\omega_Y, \pi^* \Delta, (\ba \cO_Y)^t)\big)\\
&=\big(f_* \mJ(\omega_Y, \pi^* \Delta, (\ba \cO_Y)^t)\big) \cap K(X).
\end{array}
\]
\end{itemize}
\end{lemma}

We now recall the definitions of test ideals and modules.
The reader is referred to \cite{BlickleSchwedeTakagiZhang} and \cite{SchwedeTuckerTestIdealSurvey} for details.

\begin{definition}
Suppose that $X$ is an algebraic variety over an $F$-finite field of characteristic $p > 0$.
Further, suppose that $\Delta$ is an effective $\bQ$-divisor, $\ba$ is an ideal sheaf and $t \geq 0$ is a real number.  Let $F^e : X \to X$ denote the $e^{\rm th}$ iteration of the absolute Frobenius morphism.
\begin{itemize}
\item[(a)]  We define the \textit{test ideal} $\tau(X, \Delta, \ba^t)$ to be the smallest nonzero ideal sheaf $J \subseteq \cO_X$ which satisfies the following condition locally.
For every $e \ge 0$ and every map $$\phi \in \sHom_{\cO_X}(F^e_* \cO_X(\lceil (p^e - 1)\Delta \rceil), \cO_X),$$ we have
    \[
    \phi\big(F^e_*(\ba^{\lceil t(p^e - 1) \rceil} J)\big) \subseteq J.
    \]
\item[(b)]  Suppose that $(p^{e_0} - 1)\Delta$ is Cartier for some $e_0$ (this is equivalent to requiring that $\Delta$ is $\Q$-Cartier and $p$ does not divide the index of $\Delta$).  We define the \textit{test module} $\tau(\omega_X, \Delta, \ba^t)$ to be the smallest nonzero submodule $J \subseteq \omega_X$ satisfying the following condition.
For all positive multiples $e$ of $e_0$, we have
    \[
    T^e\big(F^e_* \big(\ba^{\lceil t(p^e - 1)\rceil} \cdot \cO_X ((1-p^e) \Delta) \cdot J \big) \big) \subseteq J,
    \]
    where $T^e : F^e_* \omega_X \to \omega_X$ is Grothendieck's trace map.
\end{itemize}
\end{definition}

\begin{remark}
It is also easy to define $\tau(\omega_X, \Delta, \ba^t)$ when the index of $\Delta$ is divisible by $p$.  The above definition will be simplest for us though.
\end{remark}

Test ideals and modules are special types of Cartier modules.

\begin{definition}[Cartier modules, \cite{BlickleBockleCartierModulesFiniteness}]
Recall that a \emph{Cartier module} is a coherent $\cO_X$-module $M$ with a given map $\phi : F^e_* (M \tensor \sL) \to M$ for some line bundle $\sL$ and integer $e \ge 0$.  A Cartier module $(M, \phi)$ is called \emph{$F$-pure} if $\phi$ is surjective.
\end{definition}

Note that given a Cartier module $(M, \phi : F^e_* (M \tensor \sL) \to M)$ we can form a map
\begin{equation*}
\phi^2 : F^{2e}_* (M \tensor \sL^{1+p^e}) \xrightarrow{F^e_* (\phi \tensor \sL)} F^e_* (M \tensor \sL) \xrightarrow{\phi} M
\end{equation*}
and more generally maps
\begin{equation}
\label{eq.CompositionOfMaps}
\phi^n : F^{ne}_* (M \tensor \sL^{1+p^e + \ldots + p^{(n-1)e}}) \rightarrow \dots \rightarrow F^e_* (M \tensor \sL) \xrightarrow{\phi} M.
\end{equation}
In this way, we can construct Cartier modules $(M, \phi^n)$ for each integer $n> 0$.

\begin{definition}
\label{def.SigmaCartier}
Given a Cartier module $(M, \phi)$, we use the notation $\sigma(M, \phi)$ to denote the image \mbox{$\phi^n\big(F^{ne}_* (M \tensor \sL^{1+p^e + \ldots + p^{(n-1)e}})\big) \subseteq M$} for $n \gg 0$.  This image stabilizes by \cite{HartshorneSpeiserLocalCohomologyInCharacteristicP,Gabber.tStruc} so it is well defined.  It is the unique largest $F$-pure Cartier submodule of $(M, \phi)$.
\end{definition}

\begin{remark}
The test submodule $\tau(\omega_X, f^{a\over p^e - 1})$ is an $F$-pure Cartier module under the map $F^e_* \omega_X \xrightarrow{\times f^a} F^e_* \omega_X \xrightarrow{\Tr} \omega_X$.  In fact, it is the unique smallest nonzero Cartier submodule of $\omega_X$ with respect to this map.
\end{remark}

We now recall some other properties of test ideals and modules.

\begin{lemma}
\label{lem.PropertiesOfTest}
With notation as above:
\begin{itemize}
\item[(i)] $\tau(X, \Delta, \ba^t) = \tau(\omega_X, K_X + \Delta, \ba^t)$.
\item[(ii)]  If $D$ is a Cartier divisor then
\[
\tau(X, \Delta +D, \ba^t) = \tau(X, \Delta, \ba^t) \otimes \cO_X(-D).
\]
\item[(iii)] $($\cite{SchwedeTuckerTestIdealFiniteMaps}$)$ Suppose that $f : Y \to X$ is a finite surjective separable morphism with $Y$ normal.
Let $\Ram_{f}$ denote the ramification divisor of $f$ and $\Tr_{f} : f_* K(Y) \to K(X)$ the trace map. Then
\[
\begin{array}{rl}
 \tau(X, \Delta, \ba^t)\hspace*{-0.7em}&=\Tr_{f}\big(f_* \tau(Y, \pi^* \Delta - \Ram_f, (\ba \cO_Y)^t)\big),\\
 \tau(\omega_X, \Delta, \ba^t)\hspace*{-0.7em}&=\Tr_{f}\big(f_* \tau(\omega_Y, \pi^* \Delta, (\ba \O_Y)^t)\big).
\end{array}
\]
Furthermore if $\Tr_{f}(f_* \O_Y) = \O_X$ $($for example, if the degree $[K(Y) : K(X)])$ is not divisible by $p)$, then
\[
\hspace*{2em}
\begin{array}{rl}
 \tau(X, \Delta, \ba^t) \hspace*{-0.7em}&= \big(f_* \tau(Y, \pi^* \Delta - \Ram_f, (\ba \O_Y)^t)\big) \cap K(X),\\
 \tau(\omega_X, \Delta, \ba^t) \hspace*{-0.7em}&= \big(f_* \tau(\omega_Y, \pi^* \Delta, (\ba \O_Y)^t)\big) \cap K(X).
\end{array}
\]
\end{itemize}
\end{lemma}

\begin{remark}
By using Lemma \autoref{lem.PropertiesOfTest} (ii), we can define $\tau(X, \Delta, \ba^t)$ for non-effective $\Delta$ as follows.
Write $\Delta = E - D$ where $E$ is an effective $\Q$-divisor and $D$ is an effective Cartier divisor (locally if necessary).  Then
$\tau(X, \Delta, \ba^t) = \tau(X, E, \ba^t) \tensor \O_X(D) \subseteq K(X)$.  This is easily seen to be independent of the decomposition $\Delta = E - D$.
\end{remark}

We now recall how test ideals and modules behave under reduction to characteristic zero.  We refer to the sources \cite[Section 2.2]{MustataSrinivasOrdinary} and \cite{HochsterHunekeTightClosureInEqualCharactersticZero} for a detailed description of the process of reduction to characteristic zero.

Let $X$ be a scheme of finite type over a field $k$ of characteristic zero and $Z \subsetneq X$ be a closed subscheme.
Choosing a suitable finitely generated $\Z$-subalgebra $A$ of $k$, we can construct a scheme $X_A$ of finite type over $A$ and a closed subscheme $Z_{A} \subsetneq X_A$ such that
\[(Z_{A} \hookrightarrow X_A) \otimes_A k \cong Z \hookrightarrow X.\]
We can enlarge $A$ by localizing at a single nonzero element and replacing $X_A$ and $Z_{A}$ with the corresponding open subschemes.
Thus, applying the generic freeness \cite[(2.1.4)]{HochsterHunekeTightClosureInEqualCharactersticZero}, we may assume that $X_A$ and $Z_{A}$ are flat over $\Spec A$.
We refer to $(X_A, Z_A)$ as a \textit{model} of $(X, Z)$ over $A$.
If $Z$ is defined by an ideal sheaf $\ba$, then we denote by $\ba_A$ the defining ideal sheaf of $Z_A$.
If $Z$ is a prime divisor on $X$, then possibly enlarging $A$, we may assume that $Z_A$ is a prime divisor on $X_A$.
When $\Delta=\sum_i d_i D_i$ is a $\Q$-divisor on $X$, let $\Delta_A:=\sum_i d_i D_{i, A}$.

Given a closed point $s \in \Spec A$, we denote by $X_{s}$ (resp.~$Z_{s}$, $\ba_s$, $D_{i, s}$) the fiber of $X_A$ (resp.~$Z_{A}$, $\ba_A$, $D_{i,A}$) over $s$ and denote $\Delta_{s}:=\sum_i t_i D_{i, s}$.
Then $X_{s}$ is a scheme of finite type over the residue field $\kappa(s)$ of $s$, which is a finite field of characteristic $p(s)$.
After enlarging $A$ if necessarily, the $D_{i,s}$ are prime divisors and $\Delta_s$ is a $\Q$-divisor on $X_s$ for all closed points $s \in \Spec A$.

\begin{proposition}[\cite{HaraYoshidaGeneralizationOfTightClosure, SchwedeTakagiRationalPairs,TakagiInterpretationOfMultiplierIdeals}]
\label{prop.TestIdealsCoincidesWithMultiplierIdeals}
Suppose that $(X, \Delta, \ba^t)$ is a triple defined over a field $k$ of characteristic zero.  Suppose $A \subseteq k$ is a finitely generated $\bZ$-subalgebra over which $(X, \Delta, \ba^t)$ can be spread out to $(X_A, \Delta_A, \ba_A^t)$.
Then there exists a nonempty open subset $S \subseteq \Spec A$ $($over which for example a log resolution of $(X_A, \Delta_A, \ba_A^t)$ is defined$)$ such that for every closed point $s \in S$, if $K_X+\Delta$ is $\bQ$-Cartier, then
\[
\mJ(X, \Delta, \ba^t)_s = \tau(X_s, \Delta_s, \ba_s^t),
\]
and if $\Delta$ is $\bQ$-Cartier, then
\[
\mJ(\omega_{X}, \Delta, \ba^t)_s = \tau(\omega_{X_s}, \Delta_s, \ba_s^t).
\]
\end{proposition}

We define $F$-injective and Du Bois singularities.  For a survey of Du Bois singularities with more detailed explanations, see \cite{KovacsSchwedeDuBoisSurvey}.

\begin{definition}[$F$-injective singularities, \cite{FedderFPureRat}]
\label{def.DuBoisSingularities}
Let $X$ be a scheme of characteristic $p > 0$.  Then we say that $X$ has \emph{$F$-injective singularities} if for each $z \in X$, the Frobenius map $H^i_{z}(\O_{X,z}) \to H^i_{z}(F_* \O_{X,z})$ is injective for all $i \in \bZ$.  By local duality, in the case that $X$ is $F$-finite, this is equivalent to requiring that $\myH^{-i} F_* \omega_X^{\mydot} \to \myH^{-i} \omega_X^{\mydot}$ surjects for all $i \in \bZ$.  Again in the case that $X$ is $F$-finite, by the above observation, it is sufficient to check only $z$ closed points.

Suppose that $X$ is a scheme of finite type over a field $k$ of characteristic zero. Then we say that $X$ has \textit{dense $F$-injective type} if there exist a model of $X$ over a finitely generated $\Z$-subalgebra $A$ of $k$ and a Zariski-dense set of closed points $S \subseteq \Spec A$ such that $X_s$ has $F$-injective singularities for all $s \in S$.
\end{definition}

\begin{definition}[Du~Bois singularities, \cite{DuBoisMain,SteenBrinkDuBoisReview}]
Let $X$ be a reduced scheme of finite type over a field of characteristic zero.
Following the notation of \cite{DuBoisMain}, we say that $X$ has \emph{Du Bois singularities} if the canonical map $\O_X \to \DuBois{X}$ is a quasi-isomorphism.  Recall that $\DuBois{X} \cong \myR (\pi_{\mydot})_* \O_{X_{\mydot}}$ where $\pi_{\mydot} : X_{\mydot} \to X$ is a proper hypercover with smooth terms.
Alternately, suppose $X \subseteq T$ is an embedding of $X$ into a smooth variety $T$, $\pi : \tld T \to T$ is a log resolution of $(T, X)$ and $E = (\pi^{-1} X)_{\red}$.  In this case $\DuBois{X} \cong \myR \pi_* \O_E$ and it follows that $X$ has Du Bois singularities if and only if $\O_X \to \myR \pi_* \O_E$ is a quasi-isomorphism, \cf \cite{EsnaultHodgeTypeOfSubvarieties,SchwedeEasyCharacterization}.
\end{definition}

We also need the following lemma on affine cones over varieties with Du~Bois singularities.  It is well known to experts, but we do not know a suitable reference.  A related result which uses essentially the same computation is found in \cite[Theorem 4.4]{MaFInjectivityAndBuchsbaumSingularities}.

\begin{lemma}
\label{lem.HighAffineConeOverDuBoisIsDuBois}
Suppose that $X$ is a projective variety over a field of characteristic zero such that $X$ has Du Bois singularities.  Let $L$ be an ample divisor.  Then for all $m \gg 0$, the affine cone
\[
Y = \Spec \, R(X, mL) = \Spec \,\bigoplus_{j \in \bZ} H^0(X, \O_X(mjL))
\]
has Du Bois singularities.  In particular, every smooth projective variety has an affine cone with Du Bois singularities for some embedding into projective space.
\end{lemma}
\begin{proof}
Without loss of generality we may assume that $L$ is very ample and that $R(X, L)$ is generated in degree one.  We begin by considering what happens when $m = 1$.
Let $\pi : \tld Y \to Y$ be the blowup of the cone point with exceptional divisor $E \cong X$.  Since $\tld Y$ is an $\mathbb{A}^1$-bundle over $X$, we see that $\tld Y$ has Du Bois singularities (cf. \cite[Theorem 3.9]{DohertySingularitiesOfGenericProjection}).  We have an exact triangle:
\[
\DuBois{Y} \to \myR\pi_* \DuBois{\tld Y} \oplus k \to \myR \pi_* \DuBois{E} \xrightarrow{+1},
\]
which is just
\[
\DuBois{Y} \to \myR \pi_* \O_{\tld Y} \oplus k \to \myR \pi_* \O_E \xrightarrow{+1}.
\]
Since we want to show that $\O_Y \cong \DuBois{Y}$, it is easy to see that it suffices  to show that $\alpha : R^i \pi_* \O_{\tld Y} \to R^i \pi_* \O_{E}$ is an isomorphism for all $i > 0$, \cf \cite{SteenbrinkMixed}.
We consider the map $\alpha$ as a map of graded $R(X, L)$-modules.  Note that $[R^i \pi_* \O_{\tld Y}]_0$ is simply $H^i(X, \O_X) \cong H^i(E, \O_E)$ and that $\alpha$ is an isomorphism in degree zero.  However, since $\pi$ is an isomorphism away from the origin, $R^i \pi_* \O_{\tld Y}$ is nonzero in only finitely many degrees.  Since the formation of $R^i \pi_* \O_{\tld Y}$ is easily seen to be compatible with taking Veronese subrings of $R(X, L)$, we can replace $R(X, L)$ by $R(X, mL)$ for $m \gg 0$ and conclude that $R^i \pi_* \O_{\tld Y}$ lives only in degree zero.  Then $\alpha$ is an isomorphism and this completes the proof.
\end{proof}

We conclude by stating a key theorem from \cite{MustataSrinivasOrdinary} which is the basis for our result on test ideals.

\begin{theorem}\textnormal{(\cite[Theorem 5.10]{MustataSrinivasOrdinary})}
\label{thm.MustataSrinivasSurjectivityRestated}
Suppose that Conjecture \autoref{MS conj} holds. Let $\pi : X \to T$ be a projective morphism of schemes over $k$ with $X$ nonsingular, and let $E$ be a reduced
simple normal crossings divisor on $X$.
If $\pi_A: X_A \to T_A$ and $E_A$ are models over a finitely generated $\Z$-subalgebra $A \subseteq k$ for $\pi$ and $E$, respectively, then there exists a Zariski-dense set of closed points $S \subset \Spec A$ such that for
every $e \geq 1$ and every $s \in S$, the induced morphism
\[
{\pi_s}_* F^e_*\omega_{X_s}(E_s) \to {\pi_s}_*\omega_{X_s}(E_s)
\]
is surjective.
\end{theorem}

\section{A result of Gabber on $h$-cohomology}\label{Sec.Gabber}

The main content of this section was explained to us by O.~Gabber, \cite{GabberCommunicationWithBhatt}; we bear full responsibility for any shortcomings or errors in our exposition.

\begin{notation}
	Fix a Noetherian excellent base scheme $S$ that is an $\bF_p$-scheme for a prime number $p$. Let $\Sch$ denote the category of $S$-schemes of finite type. We use the following topologies on $\Sch$: the $h$-topology (\cite[\S 10]{SuslinVoevodsky}, \cite[\S 2]{BeilinsonpadicPeriods}), the $cdh$-topology (\cite[\S 12]{MazzaVoevodskyWeibel}, \cite[\S 3]{CortinasHaesemayerSchlichtingWeibelCyclicHomology}), the Nisnevich topology (\cite[\S 3.1]{MorelVoevodsky},\cite[\S 12]{MazzaVoevodskyWeibel}), and the fppf topology (\cite[Tag 021L]{stacks-project}). Recall that the Nisnevich topology is generated by \'etale covers with the residue field lifting property, while the $cdh$-topology is generated by Nisnevich covers and covers of the form $X' \sqcup Z \to X$, where $Z \to X$ is a closed immersion, and $X' \to X$ is an {\em abstract blowup} centered along $Z$, i.e., a proper map that is an isomorphism outside $Z$. We write $\Shv_h(\Sch)$, $\Shv_{cdh}(\Sch)$, $\Shv_\Nis(\Sch)$ and $\Shv_f(\Sch)$ for the corresponding topoi. There are obvious morphisms
\[ \xymatrix{ & \Shv_h(\Sch) \ar[dd]^{\pi} \ar[ld]_{\nu} \ar[rd]^{\mu} & \\
\Shv_{cdh}(\Sch) \ar[rd]^{\psi} & & \Shv_f(\Sch) \ar[ld]_{\eta} \\
	 	& \Shv_\Nis(\Sch) & }\]
		of topoi. The corresponding pushforward functors, at the level of sheaves as well as at the derived level, are fully faithful (as these are different topologies on the same category). Hence, we can (and will) view a $cdh$-sheaf as a Nisnevich sheaf satisfying the sheaf property for $cdh$-covers, and similarly for the other topologies. Passing to the derived level, we say that an object $K \in D(\Shv_\Nis(\Sch))$ satisfies $cdh$-descent if $K \simeq \mR \psi_* \psi^* K$ via the natural map, i.e., that $\mR \Gamma_\Nis(U,K) \simeq \mR\Gamma_{cdh}(U,\psi^* K)$ for any $U \in \Sch$ or equivalently that $K$ lies in the essential image of the fully faithful functor $D(\Shv_{cdh}(\Sch)) \to D(\Shv_\Nis(\Sch))$; we make similar definitions for $h$-descent and flat descent. Given pro-object $\{ X_i \}$ in $\Sch$ and a sheaf $F \in \Shv_\Nis(\Sch)$, we set $\mR\Gamma_\Nis( \lim_i X_i, F) = \colim \mR \Gamma_\Nis(X_i,F)$, and similarly for the other topologies; equivalently, we define $\mR \Gamma_\Nis(\lim X_i,-)$ as the derived functor of $F \mapsto \colim F(X_i)$. The main case of interest is when $\{X_i\}$ has affine transition maps (so $\lim X_i$ exists as an affine scheme), and the functor $F \mapsto \colim F(X_i)$ defines a point of one of the above topoi (so the higher derived functors vanish). For any ring $R$, we write $R_\perf = \colim R$, where the colimit takes place over the Frobenius maps on $R$; this construction sheafifies to give a presheaf $\calO_\perf$ on $\Sch$ which is an fppf (and hence Nisnevich) sheaf; we write $\calO \in \Shv_\Nis(\Sch)$ for the structure presheaf.
\end{notation}

\begin{remark}
	Our choice of the use of the $cdh$-topology is largely dictated by the proof presented below: one has powerful finiteness theorems for proving descent (essentially due to Voevodsky \cite{VoevodskyCompletelyDecomposed}, but we use \cite[Theorem 3.4]{CortinasHaesemayerSchlichtingWeibelCyclicHomology}), as well as an excellent description of points (which relies, at least philosophically, on Zariski's initial work on Riemann-Zariski spaces, see \cite[\S 3]{GoodwillieLichtenbaumCohomologicalBound}). The Nisnevich topology, on the other hand, can be easily replaced by the Zariski topology in the discussion below without a  serious cost: one must simply check Nisnevich descent for the relevant sheaves whilst proving $cdh$-descent (which is trivial: we only encounter locally quasi-coherent sheaves on $\Sch$).
\end{remark}

The main result we want to explain is:

\begin{theorem}[Gabber]\label{GabberTheorem}
The sheaf $\calO_\perf \in \Shv_\Nis(\Sch)$ is the $h$-sheafification of $\calO$. Moreover, $\calO_\perf$ satisfies $h$-descent. In particular, one has
\begin{align*}
 \mR\Gamma_h(X,\calO_\perf) \simeq \mR\Gamma_{cdh}(X,\calO_\perf) & \simeq \mR\Gamma_\Nis(X,\calO_\perf)\\
 & \simeq \colim \mR\Gamma_\Nis(X,\calO).
\end{align*}
for any $X \in \Sch$.
\end{theorem}

\begin{remark}
The noetherian and excellence assumptions in Theorem \ref{GabberTheorem} are not really necessary: modulo correct definitions for non-noetherian schemes, the $h$-cohomology of the structure sheaf always coincides with the cohomology of the perfection for arbitrary $\mathbb{F}_p$-schemes. More generally, one can also show a non-abelian analogue of this fact: the category of vector bundles over $\calO_\perf$ satisfies $h$-descent. Both these results are explained in \cite{BhattScholzeWittAffGr}.
\end{remark}

\begin{lemma}
The presheaf $\calO_\perf$ is topologically invariant, i.e., $\calO_\perf(X) \simeq \calO_\perf(Y)$ for any universal homeomorphism $f:Y \to X$.
\end{lemma}
\begin{proof}
The map $f$ is a finitely presented, and hence factors through sufficiently high iterates of Frobenii on $X$ and $Y$ (see \cite[Proposition 35]{KollarQuotientsByFinite} or for instance \cite{YanagiharaWeaklyNormal}), so the claim is clear.
\end{proof}

\begin{lemma}
Let $f:Y \to X$ be a proper surjective map with geometrically connected fibres. Then $\calO_\perf(X) \simeq \calO_\perf(Y)$.
\end{lemma}
\begin{proof}
	The Stein factorisation $Y \stackrel{a}{\to} Y' \stackrel{b}{\to} X$ satisfies: $a_* \calO_Y = \calO_{Y'}$, and $b$ is finite surjective with geometrically connected fibres. Then $b$ is a universal homeomorphism, so $\calO_\perf(X) \simeq \calO_\perf(Y') \simeq \calO_\perf(Y)$.
\end{proof}

\begin{lemma}
Let $f:Y \to X$ be a finite surjective morphism of affine schemes, and let $R \subset Y \times Y$ be the reduced subscheme underlying $Y \times_X Y$. Then the quotient $Y/R$ (in the sense of algebraic spaces) exists, and is an affine scheme. Moreover, the natural map $Y/R \to X$ is universal homeomorphism.
\end{lemma}
\begin{proof}
See \cite[Example 5]{KollarQuotientsByFinite}.
\end{proof}

\begin{lemma}
\label{lem.CohomOfOPerfVsExceptional}
	Let $f:Y \to X$ be a finite surjective morphism of affine schemes that is an isomorphism over an open $U \subset X$. Let $Z  = X - U$, and $Z' = f^{-1}(Z)$ be the induced closed subschemes. Then the following sequence is exact:
	\[ 0 \to \calO_\perf(X) \to \calO_\perf(Y) \oplus \calO_\perf(Z) \to \calO_\perf(Z') \to 0.\]
\end{lemma}
\begin{proof}
This follows from the previous lemma or from the fact that $\cO_{\perf}(X)$ is weakly normal.
\end{proof}

\begin{lemma}
	Let $f:Y \to X$ be a proper surjective map. Assume $f$ is an isomorphism outside a closed subscheme $Z \subset X$, and let $E = f^{-1}(Z)$ (with the induced scheme structure). Then the sequence
	\[ \mR\Gamma_\Nis(X,\calO_\perf) \to \mR\Gamma_\Nis(Y,\calO_\perf) \oplus \mR\Gamma_\Nis(Z,\calO_\perf) \to \mR\Gamma_\Nis(E,\calO_\perf)\]
	is a distinguished triangle.
\end{lemma}
\begin{proof}
	We may assume that $X$ and hence $Z$ is affine. It is then enough to check:
	\begin{enumerate}
	\item $H^i_\Nis(Y,\calO_\perf) \simeq H^i_\Nis(E,\calO_\perf)$ for $i > 0$.
\label{enum.pt1}
	\item The sequence
\label{enum.pt2}
		\[0 \to \calO_\perf(X) \to \calO_\perf(Y) \oplus \calO_\perf(Z) \to \calO_\perf(E) \to 0\]
		is exact.
	\end{enumerate}
For \eqref{enum.pt2}, let $Y \to X' \to X$ be the Stein factorisation, and let $Z' \subset X'$ be the inverse image of $Z$. Then the sequence in \eqref{enum.pt2} is identified with
\[ 0 \to \calO_\perf(X) \to \calO_\perf(X') \oplus \calO_\perf(Z) \to \calO_\perf(Z') \to 0\]
as the fibres of $Y \to X'$ are geometrically connected, so the exactness comes from the previous lemma.

Now we tackle \eqref{enum.pt1}.  Let $I$ be the defining ideal of $Z \subseteq X$ so that $I\cdot \O_Y$ also defines $E$.  Note that by \cite[Lemma 29.20.4(1), Tag 02O7]{stacks-project}, there is a $c$ such that if $n \geq c$ then
\[
\ker\big(H_{\Nis}^i(Y,\O_Y) \to H^i_{\Nis}(Y,\O_Y/I^n)\big) \subseteq I^{n - c} H_{\Nis}^i(Y,\O_Y)
\]
However, for $n \gg 0$ and because $i > 0$, we know that $I^{n-c}$ annihilates $H^i_{\Nis}(Y, \O_Y)$ by hypothesis.  Hence
\begin{equation}
\label{eq.InjectionOkFormalFunc}
H_{\Nis}^i(Y,\O_Y) \hookrightarrow H^i_{\Nis}(Y,\O_Y/I^n)
\end{equation}
injects.
On the other hand, by \cite[Lemma 29.20.4(3), Tag 02O7]{stacks-project}
\begin{equation}
\label{eq.ImEqualFormFunc}
\begin{array}{rlcl}
& \Image\big(& \hspace*{-1em} H_{\Nis}^i(Y, \O_Y/I^m) & \hspace*{-0.75em} \to H_{\Nis}^i(Y, \O_Y/I^n)\big)\\
 = \hspace*{-0.75em} & \Image\big(& \hspace*{-1em} H_{\Nis}^i(Y, \O_Y) & \hspace*{-0.75em} \to H_{\Nis}^i(Y, \O_Y/I^n)\big)
\end{array}
\end{equation}
for $m \gg n \gg 0$.

Next choose $n \gg 0$ and fix $e \gg 0$.  We consider the diagram below where the horizontal maps are induced by the canonical maps and the vertical maps are Frobenius (note the rows are \emph{not} exact).
\[
\xymatrix{
H_{\Nis}^i(Y, \O_Y) \ar[d]_{F^e} \ar@{^{(}->}[r] & H_{\Nis}^i(Y, \O_Y/I^n) \ar[d]_{F^e} \ar[r] & H_{\Nis}^i(Y, \O_Y/I)  \ar[d]_{F^e} \\
H_{\Nis}^i(Y, \O_Y)              \ar@{^{(}->}[r]_{\beta} & H_{\Nis}^i(Y, \O_Y/I^n)              \ar[r] & H_{\Nis}^i(Y, \O_Y/I)
}
\]
We are about to take direct limits along Frobenius (vertically) and we observe that
\[
\varinjlim_{F} H_{\Nis}^i(Y, \O_Y/I^n) = \varinjlim_{F} H_{\Nis}^i(Y, \O_Y/I).
\]
From the diagram, we see that the middle vertical map factors as
\[
 F^e : H_{\Nis}^i(Y, \O_Y/I^n) \to  H_{\Nis}^i(Y, \O_Y/I^{p^en}) \to  H_{\Nis}^i(Y, \O_Y/I^n)
\]
where the second map is the canonical one and the first is induced by Frobenius on $Y$.
Hence the image of the middle vertical map is contained in the image of $\beta$ by \eqref{eq.ImEqualFormFunc}.  An application of \eqref{eq.InjectionOkFormalFunc} then implies
\[
{
    \renewcommand{\arraystretch}{1.25}
\begin{array}{rl}
  & H^i_\Nis(Y,\calO_\perf) \\
= &  \displaystyle \varinjlim_{F} H_{\Nis}^i(Y, \O_Y)\\
= & \displaystyle\varinjlim_{F} H_{\Nis}^i(Y, \O_Y/I^n)\\
= &  \displaystyle\varinjlim_{F} H_{\Nis}^i(E, \O_Y/I)\\
= &  H^i_\Nis(E,\calO_\perf).
\end{array}
}
\]
This completes the proof.
%
\end{proof}

\begin{remark}
The proof of \eqref{enum.pt2} in the previous version of this paper was incorrect.  The problem was we misinterpreted  \cite[Corollaire 3.3.2]{EGAIII1}.  In particular, there it is written that if $\pi : Y \to X = \Spec R$ is proper, $\mathcal{F}$ a coherent sheaf on $Y$ and $I \subseteq R$ is an ideal, then for all $r \geq k \gg 0$ we have that:
\[
H^i(Y, I^{k+r} \mathcal{F}) = I^r \cdot H^i(Y, I^k \mathcal{F}).
\]
The problem is that the multiplication on the right side is \emph{not} the ordinary action of $R$ on $H^i(Y, I^k \mathcal{F})$.  Rather $I^r \cdot H^i(Y, I^k \mathcal{F})$ is viewed as a subset of $H^i(Y, I^{k+r} \mathcal{F})$.  The composition
\[
I^r \cdot H^i(Y, I^k \mathcal{F}) \subseteq H^i(Y, I^{k+r} \mathcal{F}) \to H^i(Y, I^k \mathcal{F})
\]
yields the ordinary action.

In our case, we incorrectly assumed that it was the ordinary action in \cite[Corollaire 3.3.2]{EGAIII1}.  Thus since we were taking $\mathcal{F} = \O_Y$ and $I$ so that $I^r \cdot H^i(Y, I^k \mathcal{F}) = 0$ for $r \gg 0$, we concluded that $H^i(Y, I^{k+r} \O_Y) = 0$ for $i > 0$.  This is not true.  Indeed suppose that $Y$ is a resolution of an isolated normal non-rational singularity $(X, x)$ and that $I = \langle f \rangle$ is a principal ideal vanishing at $x$.  Then $H^i(Y, \O_Y) \neq 0$ for some $i > 0$ hence $I^r \cdot H^i(Y, I^k \O_Y) = H^i(Y, f^{r+k} \O_Y) \cong H^i(Y, \O_Y) \neq 0$.  We thank Christopher Hacon and Linquan Ma for pointing out that our original proof was incorrect.
\end{remark}

\begin{lemma}
	$\calO_\perf$ is an $h$-sheaf, i.e., $$\calO_\perf \in \Shv_h(\Sch) \subset \Shv_\Nis(\Sch).$$
\end{lemma}
\begin{proof}
We first observe that $\calO_\perf$ is a $cdh$-sheaf by the lemmas above. For the rest, fix $X \in \Sch$. We will prove the sheaf axiom for $h$-covers of $X$ by induction on $\dim(X)$. If $\dim(X) = 0$, by passage to reductions, we conclude using the fact that $\calO_\perf$ is an fppf sheaf. In general, any $h$-cover of $X \in \Sch$ can be refined by one of the form $\sqcup U_i \stackrel{a}{\to} Y \stackrel{b}{\to} X$ where $a$ is a Zariski cover, and $b$ is a proper surjective generically finite morphism. As we already know Zariski descent, it suffices to show that
\[ 1 \to \calO_\perf(X) \stackrel{b^*}{\to} \calO_\perf(Y) \stackrel{p_1^* - p_2^*}{\to} \calO_\perf(Y \times_X Y) \]
is exact. There is an abstract blowup $X' \to X$ centered along a closed subset $Z \subset X$ such that the strict transform $Y' \to X'$ of $Y \to X$ is flat and surjective. This gives a diagram
\[ \xymatrix{ Y' \sqcup (Y \times_X Z) \ar[rd] \ar[r] & (Y \times_X X') \sqcup (Y \times_X Z) \ar[r] \ar[d] & Y \ar[d] \\
		& X' \sqcup Z \ar[r] & X. }\]
By induction, the commutativity of this diagram (and the flatness and surjectivity of $Y' \to X'$) immediately show that $b^*:\calO_\perf(X) \to \calO_\perf(Y)$ is injective. One also immediately deduces that a class in the kernel of $\calO_\perf(Y) \stackrel{p_1^* - p_2^*}{\to} \calO_\perf(Y \times_X Y)$ defines one in the similar object for the cover $Y' \sqcup (Y \times_X Z) \to X' \sqcup Z$, and, by induction,  descends to a class in $\calO_\perf(X') \oplus \calO_\perf(Z)$ as $\calO_\perf$ is an $h$-sheaf. A diagram chase shows that the resulting two classes on the inverse image of $Z$ in $X'$ agree, so we conclude by $cdh$-descent for the cover $X' \sqcup Z \to X$.
\end{proof}

\begin{lemma}[Cortinas, Haesemeyer, Schlichting, Weibel]
Fix $K \in D(\Shv_\Nis(\Sch))$. Assume that for an abstract blowup $f:Y \to X$ centered along some $Z \subset X$, the sequence
\[ \mR\Gamma_\Nis(X,K) \to \mR\Gamma_\Nis(Y,K) \oplus \mR\Gamma_\Nis(Z,K) \to \mR\Gamma_\Nis(f^{-1}Z,K)  \]
is a distinguished triangle. Then $K$ satisfies $cdh$-descent.
\end{lemma}
\begin{proof}
See \cite[Theorem 3.4]{CortinasHaesemayerSchlichtingWeibelCyclicHomology}.
\end{proof}

\begin{lemma}
The $cdh$-sheafification of $X \mapsto H^i_h(X,\calO_\perf)$ is $0$ for $i > 0$.
\end{lemma}
\begin{proof}
	By work of Goodwillie and Lichtenbaum \cite[Proposition 2.1 and Corollary 3.8]{GoodwillieLichtenbaumCohomologicalBound}, the $rh$-topology has enough points given by spectra of valuation rings, viewed as suitable pro-objects in $\Sch$. As the $cdh$-topology is generated by the $rh$-topology as Nisnevich covers, it follows that henselian valuation rings provide enough points for the $cdh$-topology.  In particular, it suffices to show that $H^i_h(-,\calO_\perf)$ vanishes on spectra of valuation rings for $i > 0$. For a valuation ring $R$, there is an ind-(finite flat) extension $R \to S$ where $S$ is a valuation ring with an algebraically closed fraction field. By flat descent for $\calO_\perf$-cohomology (together with the observation that $\Spec(S) \to \Spec(R)$ is a cofiltered limit of $h$-covers), it suffices to show that $\mR\Gamma_h(\Spec(S),\calO_\perf) = S_\perf$, which is clear: $S$ is a point object for the $h$-topology, i.e., any $h$-cover $Y \to S$ has a section (see \cite[Proposition 2.2]{GoodwillieLichtenbaumCohomologicalBound}).
\end{proof}

\begin{proof}[Proof of Theorem \autoref{GabberTheorem}]
	We already know $\calO_\perf$ is an $h$-sheaf. To identify this $h$-sheaf as the $h$-sheafification of $\calO$, it suffices to observe that the relative Frobenius morphism $X \to X^{(1)}$ for any $X \in \Sch$, which is an $h$-cover, becomes an isomorphism after $h$-sheafification; this, in turn, follows by identifying the reduced subscheme of $X \times_{X^{(1)}} X$ as the diagonal $X \subset X \times_{X^{(1)}} X$ and the fact the passing to reduced subschemes is an isomorphism after $h$-sheafification. Hence, $\calO_\perf$ is the $h$-sheafification of the structure presheaf $\calO$. The last lemma above shows $\mR \nu_* \calO_\perf \simeq \calO_\perf$, while $cdh$-descent shows $\mR \psi_* \calO_\perf \simeq \calO_\perf$.
\end{proof}

\begin{corollary}
\label{thm.GabberIsomorphism}
Suppose $X \in \Sch$ is affine.  Then
\[
	\lim_{\rightarrow} \mR\Gamma(X_\mydot, \O_{X_\mydot}) \simeq \O_{X}^{1/p^{\infinity}}(X) = \O_{\perf}(X).
\]
Here the limit runs over all proper hypercovers $\pi_{\mydot} : X_{\mydot} \to X$.
\end{corollary}

\begin{remark}
When the base scheme $S$ is the spectrum of a field, by de Jong's alterations theorem, the colimit in Corollary \ref{thm.GabberIsomorphism} can be taken over all proper hypercovers with each $X_n$ smooth.
\end{remark}

\begin{proof}
	By Verdier's hypercovering theorem, the left hand side is identified with $\mR\Gamma_h(X,\calO_\perf)$. By Theorem \ref{GabberTheorem}, this is identified with $\mR\Gamma_\Nis(X,\calO_\perf)$. If $X$ is affine, then the higher Nisnevich (or even fppf) cohomology of quasicoherent sheaves vanishes, so this complex identified with $\calO_\perf(X)$.
\end{proof}

\section{Du Bois versus $F$-injective singularities}

Our goal in this section is first to use the methods developed in \cite{MustataSrinivasOrdinary} to prove a relationship between $F$-injective and Du~Bois singularities.
\begin{conjecture}\label{DB conj}
Let $X$ be a reduced scheme of finite type over a field of characteristic zero.
Then $X$ has Du Bois singularities if and only if $X$ is of dense $F$-injective type.
\end{conjecture}

Indeed, we will show that:

\begin{theorem}
\label{thm.DuBoisImpliesFInjectiveType}
Conjecture \ref{MS conj} holds if and only if Conjecture \ref{DB conj} holds as well.
\end{theorem}

Before proving this, we must first show that Conjecture \autoref{MS conj} actually implies that Frobenius acts bijectively on all cohomology groups of the structure sheaf.  We then use this to generalize Theorem \autoref{thm.MustataSrinivasSurjectivityRestated}.

\begin{lemma}\label{lemma.every j}
Suppose that Conjecture \autoref{MS conj} holds.  Let $V$ be an $n$-dimensional smooth projective variety over a field $k$ of characteristic zero.
Given a model of $V$ over a finitely generated $\Z$-subalgebra $A$ of $k$, there exists a Zariski-dense set of closed points $S  \subseteq \Spec A$ such that the action induced by Frobenius on $H^i(V_s, \mathcal{O}_{V_s})$ is bijective for every $s \in S$ and $0 \leq i \leq n$.
\end{lemma}
\begin{proof}
We claim that there is a finite set of varieties, such that if they simultaneously satisfy the condition of Conjecture \autoref{MS conj} for some $S \subseteq \Spec A$,  then the condition of the lemma holds for $V$.  Note that again it is harmless to assume that $k$ is algebraically closed as in Remark \autoref{rem.FinitelyManySmoothIsOk}.
We prove this claim by strong induction on the dimension $n$.
Let $D$ be a smooth ample divisor on $V$, chosen sufficiently positive so that $H^i(V, \O_V(-D))= 0$ for all $i \le n-1$. Then $H^i(V, \O_V) \to H^i(D, \O_D)$ is an isomorphism for all $i \leq n - 2$ and an injection for $i = n-1$.  Since $\dim D = n-1$, we can apply our induction hypothesis to $D$.  It immediately follows that Frobenius acts injectively (and hence bijectively) on $H^i(V_s, \O_{V_s})$ for $i \leq n-1$ for appropriate $S$ and all $s \in S$.  Of course, we may also certainly assume that Frobenius acts bijectively on $H^n(V_s, \O_{V_s})$ for all $s \in S$ by Remark \autoref{rem.FinitelyManySmoothIsOk}.
This completes the proof.
\end{proof}

We now state our promised generalization of \cite[Theorem 5.10]{MustataSrinivasOrdinary}, \ie{} of Theorem \autoref{thm.MustataSrinivasSurjectivityRestated}.

\begin{theorem}
\label{theorem.SurjectsAllCohomologySNC}
With notation as in Theorem \autoref{thm.MustataSrinivasSurjectivityRestated}, we additionally have that
\[
R^i {\pi_s}_* F^e_*\omega_{X_s}(E_s) \to R^i {\pi_s}_*\omega_{X_s}(E_s)
\]
is surjective for every $i \geq 0$.
\end{theorem}
To prove this, we follow the method of \cite{MustataSrinivasOrdinary}, first proving several lemmas.

\begin{lemma}\textnormal{(\cite[Lemma 5.6]{MustataSrinivasOrdinary})}\label{lemma.MS lemma5.6}
Suppose that $E = \bigcup_{i \in I} E_i$ is a projective simple normal crossings variety over a field $k$ of characteristic zero, and that we are given a model of $E$ over a finitely generated $\Z$-subalgebra $A$ of $k$.
Assuming Conjecture \autoref{MS conj}, there exists a Zariski-dense set of closed points $S \subset \Spec A$ such that the Frobenius action $F : H^j(E_s, \O_{E_s}) \to H^j(E_s, \O_{E_s})$ is bijective for every $s \in S$ and $j \ge 0$.
\end{lemma}
\begin{proof}This is essentially taken from \cite[Lemma 5.6]{MustataSrinivasOrdinary}.  Without loss of generality, we may assume that for all closed points $s \in \Spec A$, $E_s$ is still simple normal crossings, in particular each $(E_i)_s$ is smooth and even more every intersection $(E_{I'})_s := \bigcap_{i \in I'} (E_i)_s$ is also smooth over $k(s)$ for all $I' \subseteq I$.
We apply Conjecture \autoref{MS conj} (and Lemma \autoref{lemma.every j}) to each component of these intersections and so assume that the Frobenius action is bijective on each of their cohomology groups for some Zariski-dense set $S \subset \Spec A$.
Fix an arbitrary $s \in S$.
We denote by $n-1$ the dimension of $E$ and form an acyclic complex
\[
C^{\mydot} : 0 \to C^0 \xrightarrow{d^0} C^1 \xrightarrow{d^1} \ldots \xrightarrow{d^{n-1}} C^n \xrightarrow{d^n} 0.
\]
where $C^r = \bigoplus_{|I'| = r} \O_{(E_{I'})_s}$ and $C^0 = \O_{E_s}$.  Note that if $|I'| = r$, then $E_{I'} = \bigcap E_i$ has dimension $n - r$ since intersecting $r$ components of $E$ results in a scheme of dimension $n-r$.
We put $Z_i = \ker d_i$ and observe that we have exact sequences compatible with Frobenius for each $r \ge 1$ and $j \ge 1$:
\begin{align*}
& H^{j-1}(X_s, C^r) \to
H^{j-1}(X_s, Z^{r+1}) \to
H^{j}(X_s, Z^{r}) \\
\to & H^{j}(X_s, C^r) \to
H^{j}(X_s, Z^{r+1}).
\end{align*}
Note that Frobenius is bijective on the $H^{j}(X_s, C^r)$ terms for any $r \ge 1$ and $j \ge 0$.  We then perform a descending induction on $r$.  When $r=n$, it is obvious that Frobenius acts bijectively on $H^{j}(X_s, Z^{r})$ for all $j \ge 0$, hence the base case is taken care of.  Thus the Frobenius acts bijectively on $H^{j}(X_s, Z^{r})$ for any $r \ge 1$ and $j \ge 0$ (say by the 5-lemma or \cite[Lemma 2.4]{MustataSrinivasOrdinary}).
Since $Z^1 = \Image{d^0} \cong \O_{E_s}$, we complete the proof.
\end{proof}

Continuing to follow \cite{MustataSrinivasOrdinary} we have the following.
\begin{lemma}\textnormal{(\cite[Corollary 5.7]{MustataSrinivasOrdinary})}
\label{lem.SemiSimpleFrobMinusSNC}
With notation as in Lemma \autoref{lemma.MS lemma5.6}, suppose that $E \subset X$ is a reduced simple normal crossings divisor in a smooth projective variety $X$.
Assuming Conjecture \autoref{MS conj}, there exists a Zariski-dense set of closed points $S \subset \Spec A$ such that
\[
F : H^i(X_s, \O_{X_s}(-E_s)) \to H^i(X_s, \O_{X_s}(-E_s))
\]
is bijective for every $s \in S$ and $i \ge 0$.
\end{lemma}
\begin{proof}
This immediately follows from an application of the five lemma (or \cite[Lemma 2.4]{MustataSrinivasOrdinary}) to the cohomology sequence of the short exact sequence $0 \to \O_{X_s}(-E_s) \to \O_{X_s} \to \O_{E_s} \to 0$ and the fact that the Frobenius actions on $H^i(X_s, \O_{X_s})$ and $H^i(E_s, \O_{E_s})$ are bijective.
\end{proof}

We now come to the proof of Theorem \autoref{theorem.SurjectsAllCohomologySNC}.

\begin{proof}[Proof of Theorem \autoref{theorem.SurjectsAllCohomologySNC}]
The proof is essentially the same as that of \cite[Theorem 5.10]{MustataSrinivasOrdinary}.  Note that as observed in Remark \autoref{rem.FinitelyManySmoothIsOk} we may assume that $k$ is algebraically closed.
By the same argument as that of \cite[Theorem 5.10]{MustataSrinivasOrdinary}, we can reduce to the case where $X$ and $T$ are projective.
We choose a sufficiently ample $\sL$ on $T$ such that each of the direct image sheaves $(R^i \pi_*\omega_{X}(E)) \tensor \sL$ is globally generated, 
and such that the higher cohomology $H^j(T, (R^i \pi_* \omega_X(E)) \tensor \sL) = 0$ for all $j \ge 1$.
By Bertini's theorem, we then choose a general section $D' \in |\sL|$ so that $E+\pi^*D'$ is a reduced simple normal crossing divisor.
Put $E'=\pi^*D'$ and observe that we have surjections
\[
H^i(X_s, F_* \omega_{X_s}(E_s + E'_s)) \to H^i(X_s, \omega_{X_s}(E_s + E'_s))
\]
by Serre duality and Lemma \autoref{lem.SemiSimpleFrobMinusSNC} (applied to $E + E'$) for an appropriate Zariski-dense set of closed points $S \subseteq \Spec A$ and all $s \in S$.  Of course, these maps factor through surjections:
\[
\xymatrix{
H^i(X_s, F_* \omega_{X_s}(E_s + p E'_s)) \ar@{-}[d]^{\sim} \ar@{->>}[r] & H^i(X_s, \omega_{X_s}(E_s + E'_s)) \ar@{-}[d]^{\sim} \\
H^i(X_s, (F_* \omega_{X_s}(E_s)) \tensor \pi_s^*\sL_s ) \ar@{->>}[r] & H^i(X_s, \omega_{X_s}(E_s) \tensor \pi_s^* \sL_s).
}
\]
Since we may assume that $H^j(T_s, (R^i {\pi_s}_* \omega_{X_s}(E_s)) \tensor \sL_s) = 0$ for every closed point $s \in \Spec A$ and $j \ge 1$,
by a simple spectral sequence argument, 
the following maps surject for all $s \in S$:
\[
H^0(T_s, (F_* R^i {\pi_s}_* \omega_{X_s}(E_s)) \tensor \sL_s) \to H^0(T_s, (R^i {\pi_s}
_* \omega_{X_s}(E_s)) \tensor \sL_s).
\]
By the global generation we claimed in characteristic zero, which is preserved on an open dense set of points in $\Spec A$, the proof is complete.
\end{proof}

We can now prove our theorem relating $F$-injective and Du Bois singularities after reduction to characteristic $p > 0$.
\begin{proof}[Proof of Theorem \autoref{thm.DuBoisImpliesFInjectiveType}]
We first show that assuming Conjecture \autoref{MS conj}, if $Y$ has Du Bois singularities, then $Y$ has dense $F$-injective type.
Indeed, suppose that $Y$ is a reduced scheme over $k$ in characteristic zero with Du Bois singularities.  Working locally if necessary, embed $Y$ inside a smooth variety $T$ and let $\pi : X \to T$ be a projective log resolution of $(T, Y)$ with $E = (\pi^{-1} Y)_{\red}$.  We spread this data out to a model over a finitely generated $\bZ$-subalgebra $A$ of $k$.
We may restrict our attention to $s \in \Spec A$ such that $R^i {\pi_s}_* \omega_{X_s} = 0$ for all $i > 0$ by Grauert--Riemenschneider vanishing \cite{GRVanishing}.
Note that then $R^i {\pi_s}_* F^e_*\omega_{X_s}(E_s) \rightarrow R^i {\pi_s}_* F^e_* \omega_{E_s}$ surjects for all $i \geq 0$.
By Theorem \autoref{theorem.SurjectsAllCohomologySNC}, there exists a Zariski-dense set of closed points $S  \subseteq \Spec A$ such that
\begin{equation}
\label{eq.SurjectionImplies}
R^i {\pi_s}_* F^e_* \omega_{E_s} \to R^i {\pi_s}_* \omega_{E_s}
 \end{equation}
 surjects for all $i \geq 0$ and $s \in S$.  However, since $Y$ is Du Bois, we have
$R^i {\pi_s}_* \omega_{E_s} \cong \myH^{i-\dim E} \omega_{Y_s}^{\mydot}.$  This combined with the surjection of \eqref{eq.SurjectionImplies} implies that $Y_s$ is $F$-injective by local duality.

Now we assume that every variety with Du Bois singularities has dense $F$-injective type and we  prove Conjecture \autoref{MS conj}.  Let $X$ be a smooth projective variety over a field of characteristic zero.
Fix an ample divisor $B$ on $X$ and consider the affine cone $Y = \Spec R(X, mB)$ for $m \gg 0$.
We may restrict our attention to $s \in \Spec A$ such that $R(X, mB)_s=R(X_s, mB_s)$.
Note that $Y$ is Du Bois by Lemma \autoref{lem.HighAffineConeOverDuBoisIsDuBois}.
It follows by assumption that $Y_s$ is $F$-injective for all $s$ in some dense set of maximal ideals $S \subseteq \Spec A$.
In particular, the Frobenius map
\begin{equation}
\label{eq.LocalCohomologyInjects}
F : H^{n+1}_{y_s}(\O_{Y_s}) \to H^{n+1}_{y_s}(\O_{Y_s})
 \end{equation}
injects where $y$ is the cone point.  Since $R(X, mB)_s=R(X_s, mB_s)$ is graded, \eqref{eq.LocalCohomologyInjects} injects in degree zero.
However, this is nothing but requiring that $F : H^n(X_s, \O_{X_s}) \to H^n(X_s, \O_{X_s})$ injects, and hence bijects, which is what we wanted to prove.
\end{proof}

\begin{remark}
It is worth pointing out that a weaker version of Conjecture \autoref{DB conj}, \emph{only} for isolated standard graded singularities, is all that is needed for the ($\Leftarrow$) direction of Theorem \autoref{thm.DuBoisImpliesFInjectiveType}.  This follows from the proof above since the only singularities we consider are cones over smooth varieties.
\end{remark}

We now come to our other main result on Du~Bois and $F$-injective singularities.  We utilize the work of Gabber from Section \autoref{Sec.Gabber}.

\begin{theorem}
\label{thm.DuBoisFinjectiveUnified}
Suppose that $X$ is a reduced scheme of finite type over a field $k$.  Then for any proper hypercover with smooth terms $\pi_{\mydot} : X_{\mydot} \to X$, consider the canonical map
\begin{equation}
\label{eq.LocalCohomologyDuBoisThmC2}
H^i_{z}(\O_{X,z}) \to \bH^i_z((\myR (\pi_{\mydot})_* \O_{X_{\mydot}})_z)
\end{equation}
for each $i$ and each point $z \in X$.
\begin{itemize}
\item{} Suppose $k$ is of characteristic zero, then $X$ has Du Bois singularities if and only if $\eqref{eq.LocalCohomologyDuBoisThmC2}$ injects for each $\pi_{\mydot}$, $i$ and $z$.
\item{} Suppose $k$ is of characteristic $p > 0$ and is $F$-finite, then $X$ has $F$-injective singularities if and only if  $\eqref{eq.LocalCohomologyDuBoisThmC2}$ injects for each $\pi_{\mydot}$, $i$ and $z$.
\end{itemize}
\end{theorem}
\begin{proof}
There is nothing to prove in characteristic zero by \cite{KovacsDuBoisLC1}.  In characteristic $p > 0$, suppose first that $X$ has $F$-injective singularities.  It follows easily that for every point $z \in X$ and every $i \geq 0$ that
\[
H^i_z(\O_{X,z}) \hookrightarrow \lim_{e \shortrightarrow  \infty}H^i_z(\O_{X,z}^{1/p^e}) \cong H^i_z\big(\lim_{e \shortrightarrow \infty} \O_{X,z}^{1/p^e}\big) = H^i_z(\O_{X,z}^{1/p^{\infty}})
\]
injects.  Now, choose $\pi_{\mydot} : X_{\mydot} \to X$ a proper hypercover with smooth terms of $X$.
Then for any $i \geq 0$ and point $z \in X$, we obtain a factorization of the above map
\[
H^i_z(\O_{X,z}) \to \bH^i_z((\myR (\pi_{\mydot})_* \O_{X_{\mydot}})_z) \to H^i_z(\O_{X,z}^{1/p^{\infty}})
\]
by Theorem \autoref{thm.GabberIsomorphism}. Hence \eqref{eq.LocalCohomologyDuBoisThmC2} injects as claimed.

The converse is even easier, note that if $\pi_{\mydot} : X_{\mydot} \to X$ is a proper hypercover with smooth terms, then $\mu_{\mydot} : X_{\mydot} \xrightarrow{\pi_{\mydot}} X \xrightarrow{F} X$ is also a proper hypercover with smooth terms.  Hence if $H^i_{z}(\O_{X,z}) \to \bH^i_z( (\myR (\mu_{\mydot})_* \O_{X_{\mydot}})_z)$ injects, so does $H^i_z(\O_{X,z}) \to H^i_z(F_* \O_{X,z})$ which proves that $X$ is $F$-injective.
\end{proof}

\section{Comparison between multiplier ideals and test ideals}

Before proving Theorem A from the introduction, we prove a lemma which is essentially the heart of the proof.
First we make a small definition.

\begin{definition}
Suppose that $(M, \phi : F^e_* (M \otimes \sL) \to M)$ is a Cartier module on a scheme $X$.  Further suppose that $D$ is an effective Cartier divisor on $X$.
Since $D$ is effective, we have an induced map
 \[
 \phi_D : F^e_* (M \otimes \sL \otimes \O_X(-(p^e - 1)D)) \subseteq F^e_* (M \otimes \sL) \xrightarrow{\phi} M.
 \]
Setting $\sN = \sL \otimes \O_X( -(p^e-1)D)$ we see that $(M, \phi_D : F^e_* (M \otimes \sN) \to M)$ is also a Cartier module.  Then we use $\sigma(M, \phi, D) = \sigma(M, \phi_D)$ to denote the maximal $F$-pure Cartier submodule of $(M, \phi_D)$ as in Definition \autoref{def.SigmaCartier}.
\end{definition}

\begin{lemma}
\label{lem.KeyCase}
Suppose that $X = \Spec R$ is a normal affine variety over a field $k$ of characteristic zero, $f \in R$ is a nonzero element with $D = \Div(f)$ and $a_1 \ge 1$ is an integer.
Suppose also that we are given a model of $(X, f)$ over a finitely generated $\bZ$-subalgebra $A$ of $k$.
Let $a_0$ be the largest jumping number of $\mJ(\omega_X, f^t)$ which is less than $a_1$.
Then there exist finitely many smooth projective varieties $X^{(i)}$ over $k$ and a nonempty open subset $U \subseteq \Spec A$ satisfying the following.
If $s \in U$ is a closed point and the Frobenius maps
\begin{equation}
\label{eq.IsomorphismCohomologyLemma}
F: H^{\dim X^{(i)}_{s}} (X^{(i)}_{s}, \O_{X^{(i)}_{s}}) \to H^{\dim X^{(i)}_{s}} (X^{(i)}_{s}, \O_{X^{(i)}_{s}})
\end{equation}
are bijective for all $i$, then $\mJ(\omega_{X}, f^{a})_s = \tau(\omega_{X_s}, f_s^{a})$ for all $a_0 \leq a < a_1$.
\end{lemma}
\begin{proof}
Throughout this proof, we let $p(s)$ denote the characteristic of the residue field $k(s)$ of a closed point $s \in \Spec A$.
We fix a log resolution $\pi : Y \to X$ of $(X, f)$ and let $F = (\Div_Y(f))_{\red}= (\pi^{*} D)_{\red}$.
It follows from the proof of \cite[Theorem 5.10]{MustataSrinivasOrdinary} that there exist finitely many smooth projective varieties $X^{(i)}$ over $k$ and a nonempty open subset $U \subseteq \Spec A$
such that the bijections of \eqref{eq.IsomorphismCohomologyLemma} above guarantee that the log trace maps
\begin{equation*}
F^e_* {\pi_s}_* \O_{Y_s}(K_{Y_s} + F_{s}) \to {\pi_s}_* \O_{Y_s}(K_{Y_s} + F_{s})
\end{equation*}
surject for each $e > 0$.
By the projection formula, we obtain the surjectivity of the following maps
\begin{equation*}
F^e_* {\pi_s}_* \O_{Y_s}(K_{Y_s} + F_{s} -  p(s)^e a_1 \pi^*D_s) \to {\pi_s}_* \O_{Y_s}(K_{Y_s} + F_{s} - a_1 \pi^* D_s),
\end{equation*}
which we identify with
\begin{equation}
\label{eq.inLemmaSurjectivity}
\begin{array}{rl}
\Tr^e_{a_1 D_s} : & F^e_* {\pi_s}_* \O_{Y_s}(K_{Y_s} + F_{s} - a_1 \pi^*  D_s) \\
& \hspace*{3.6em} \xrightarrow{\Tr^e(F^e_* f_s^{a_1(p(s)^e - 1)} \cdot \blank)}  {\pi_s}_* \O_{Y_s}(K_{Y_s} + F_{s} - a_1 \pi^* D_s).
\end{array}
\end{equation}
Note that the map of \eqref{eq.inLemmaSurjectivity} is the restriction of $$\Tr^e_{a_1 D_s} : F^e_* \omega_{X_s} \xrightarrow{\Tr^e(F^e_* f_s^{a_1(p(s)^e - 1)}\cdot \blank) } \omega_{X_s}$$ which can also be identified with $\Tr^e : F^e_* (f_s^{a_1(p(s)^e-1)} \omega_{X_s}) \to \omega_{X_s}$.
Shrinking $U$ if necessary and applying Proposition \autoref{prop.TestIdealsCoincidesWithMultiplierIdeals},  we can assume that the following holds for all closed points $s \in U$:
\begin{itemize}
\item[(a)]  $\pi_s: Y_s \to X_s$ is a log resolution of $(X_s, f_s)$.
\item[(b)]  $\mJ(\omega_{X}, f^a)_s= \mJ(\omega_{X_s}, f_s^{a})={\pi_s}_* \O_{Y_s}(\lceil K_{Y_s}- a \pi_s^*  D_s \rceil)$ for all $a \in \R_{\ge 0}$.
\item[(c)]  $\mJ(\omega_{X})_s = \tau(\omega_{X_s})$.
\end{itemize}
By (b), we obtain that for sufficiently small $\epsilon>0$,
\begin{equation}
\label{eq.ZerothContainment}
\begin{array}{rl}
{\pi_s}_* \O_{Y_s}(K_{Y_s} + F_{s} - a_ 1 \pi^* D_s)
=  \mJ(\omega_{X_s}, f_s^{a_1-\epsilon})
= & \mJ(\omega_{X}, {f}^{a_1-\epsilon})_s\\
= & \mJ(\omega_{X}, {f}^{a_0})_s.
 \end{array}
\end{equation}

We now fix an arbitrary closed point $s \in U$.
For all $a_0 \leq a < a_1$, we have
\begin{equation}
\label{eq.FirstContainment}
 \tau(\omega_{X_s}, f_s^{a})
\subseteq  \big({\pi_s}_* \O_{Y_s}(\lceil K_{Y_s} - a \pi_s^* D_s \rceil) \big)
=  \mJ(\omega_{X}, {f}^{a})_s,
\end{equation}
where the first containment follows from the universal property defining $\tau(\omega_{X_s}, f_s^{a})$.\footnote{$\tau(\omega_{X_s}, f_s^{a})$ is the smallest nonzero submodule of $\omega_X$ stable under $\Tr^e(F^e_* f^{\lceil a(p^e - 1) \rceil} \blank)$ for all $e > 0$.}  This is the ``easy'' containment, see for example the argument of \cite{SmithFRatImpliesRat} or \cite[Proposition 4.2]{MustataSrinivasOrdinary}.
On the other hand, $(\tau(\omega_{X_s}), \Tr : F^e_* \omega_{X_s} \to \omega_{X_s})$ is a Cartier module which implies that $$\Tr^e(F^e_* (f_s^{(p(s)^e - 1)a_1} \tau(\omega_{X_s}))) \subseteq \Tr^e(F^e_* \tau(\omega_{X_s})) \subseteq \tau(\omega_{X_s}).$$
Hence we obtain a Cartier-module structure on $\tau(\omega_{X_s})$ by restricting same map $\Tr^1_{a_1 D_s}$ we studied immediately after \eqref{eq.inLemmaSurjectivity}.
It follows then that for any $a_0 \leq a < a_1$
\begin{equation}
\label{eq.SecondContainment}
\begin{array}{rl}
\tau(\omega_{X_s}, f_s^{a})
=  & \sum_{e \gg 0} \Tr^e\big(F^e_* (f_s^{\lceil p(s)^e a \rceil} \tau(\omega_{X_s}))\big)  \\
\supseteq & \sum_{e \gg 0} \Tr^e\big(F^e_* (f_s^{(p(s)^e - 1)a_1} \tau(\omega_{X_s}))\big) \\
= & \sigma( \tau(\omega_{X_s}), \Tr_{a_1 D_s}),
\end{array}
\end{equation}
where the first equality follows from \cite[Lemma 3.21]{BlickleSchwedeTakagiZhang}.

Now \eqref{eq.inLemmaSurjectivity} and \eqref{eq.ZerothContainment} imply that $\mJ(\omega_{X}, {f}^{a_0})_s = \mJ(\omega_{X}, {f}^{a})_s  $ is an $F$-pure Cartier module under the action of $\Tr_{a_1 D_s}$.  It is certainly contained in $\tau(\omega_{X_s})$ by (c).  Hence, since $\sigma( \tau(\omega_{X_s}), \Tr_{a_1 D_s})$ is the largest $F$-pure Cartier module contained in $\tau(\omega_{X_s})$, we have
\begin{equation}
\label{eq.ThirdContainment}
\mJ(\omega_{X}, {f}^{a})_s   \subseteq \sigma( \tau(\omega_{X_s}), \Tr_{a_1 D_s}).
\end{equation}
Combining \eqref{eq.ThirdContainment}, \eqref{eq.SecondContainment} and \eqref{eq.FirstContainment}, we obtain for any $a_0 \leq a < a_1$ that
\[
\begin{array}{rl}
\mJ(\omega_{X}, {f}^{a})_s \subseteq \sigma( \tau(\omega_{X_s}), \Tr_{a_1 D_s}) \subseteq \tau(\omega_{X_s}, f_s^{a}) \subseteq \mJ(\omega_{X}, {f}^{a})_s
\end{array}
\]
which completes the proof of the lemma.
\end{proof}

Now we come to the theorem from the introduction.
\begin{theorem}
\label{thm.EquivalentConjectures}
Conjecture \autoref{test ideal conj} holds if and only if Conjecture \autoref{MS conj} holds.
\end{theorem}
\begin{proof}
It is proved by \cite{MustataOrdinary2} that if Conjecture \autoref{test ideal conj} holds, then Conjecture \autoref{MS conj} holds as well, so we consider the converse.  We recall the setup of Conjecture \autoref{test ideal conj}.
Let $X$ be a normal variety over a field $k$ of characteristic zero.
Suppose that $\Delta$ is a $\Q$-divisor on $X$ such that $K_X+\Delta$ is $\Q$-Cartier, and $\ba$ is a nonzero ideal on $X$.
Suppose also that we are given a model of $(X, \Delta, \ba)$ over a finitely generated $\Z$-subalgebra $A$ of $k$.
Under these hypotheses, we need to show that there exists a Zariski-dense set of closed points $S  \subseteq \Spec A$ such that
\begin{equation}\label{eq.testVsMult}
\tau(X_s, \Delta_s, \ba_s^{\lambda})=\mathcal{J}(X, \Delta, \ba^{\lambda})_s \textup{ for all $\lambda \in \R_{\ge 0}$ and all $s \in S$.}
\end{equation}

Since proving Conjecture \ref{test ideal conj} for $(X, \Delta, \ba)$ is equivalent to proving it simultaneously for all $(U_i, \Delta |_{U_i}, \ba |_{U_i})$, where $X=\bigcup_i U_i$ is a finite affine open cover of $X$, it is enough to consider the case when $X$ is affine.  Likewise as in Remark \autoref{rem.FinitelyManySmoothIsOk} we may assume that $k$ is algebraically closed since the formation of the multiplier ideal commutes with field base change.
Arguing exactly as in \cite[Proposition 4.3]{MustataSrinivasOrdinary}, one may also assume that $\ba = \langle f \rangle$ is a principal ideal.

Now, let $\alpha : W \to X$ be a finite cover from a normal $W$ such that $\alpha^* (K_X + \Delta)$ is an integral Cartier divisor.
After possibly enlarging $A$, we may assume that $\alpha_s:W_s \to X_s$ is a finite separable cover with $W_s$ normal and $\alpha_s^* (K_{X_s}+\Delta_s)$ is Cartier for all closed points $s \in \Spec A$.
Note that by Lemmas \ref{lem.PropertiesOfMultiplier} (iii) and \ref{lem.PropertiesOfTest} (iii),
\begin{align*}
\mJ(X, \Delta, f^{\lambda}) =& \mJ(\omega_X, K_X + \Delta, f^{\lambda})\\
=& \Tr_{W/X}\big(\alpha_* \mJ(\omega_W, \alpha^*(K_X + \Delta), f^{\lambda})\big),\\
\tau(X_s, \Delta_s, f_s^{\lambda})=& \tau(\omega_{X_s}, K_{X_s} + \Delta_s, f_s^{\lambda})\\
=& \Tr_{W_s/X_s}\big({\alpha_s}_* \tau(\omega_{W_s}, \alpha_s^*(K_{X_s} + \Delta_s), f_s^{\lambda})\big).
\end{align*}
Since $\alpha^*(K_X + \Delta)$ (resp. $\alpha_s^*(K_{X_s} + \Delta_s)$) is Cartier, by Lemma \ref{lem.PropertiesOfMultiplier} (ii) (resp. \ref{lem.PropertiesOfTest} (ii)), we can pull it out from the multiplier (resp. test) module.
In particular, instead of showing
\eqref{eq.testVsMult},
it suffices to prove that there exists a Zariski-dense set of closed points $S  \subseteq \Spec A$ such that
\begin{equation}
\label{test submod=multsubmod}
\tau(\omega_{X_s}, f_s^{\lambda})=\mathcal{J}(\omega_X, f^{\lambda})_s \textup{ for all $\lambda \in \R_{\ge 0}$ and all $s \in S$.}
\end{equation}
Since both $\tau(\omega_{X_s}, f_s^{\lambda})$ and $\mJ(\omega_X, f^{\lambda})$ are periodic in $\lambda$ with period $1$, it is enough to show \eqref{test submod=multsubmod} for all $\lambda \in [0, 1]$.

Next, we let $0 = \lambda_0 < \lambda_1 < \ldots < \lambda_n = 1$ denote the distinct jumping numbers of $\mJ(\omega_X, f^{\lambda})$ in $[0,1]$.
Note that the $\lambda_j$ are rational numbers.

\begin{claim}\label{keyclaim}
There exist a nonempty open subset $U \subseteq \Spec A$ and finitely many smooth projective varieties $X^{(i)}_{j}$ over $k$ for each $\lambda_j$ satisfying the following.
If $s \in U$ is a closed point and the action of Frobenius on $H^{\dim X^{(i)}_{j,s}}(X^{(i)}_{j,s}, \O_{X^{(i)}_{j,s}})$ is bijective for all $i$, then $\tau(\omega_{X_s}, f_s^{\lambda})=\mathcal{J}(\omega_X, f^{\lambda})_s$ for all $\lambda_{j-1} \leq \lambda < \lambda_j$.
\end{claim}
\begin{proof}[Proof of Claim \ref{keyclaim}]
Write $\lambda_j = a_1/m$ with $a_1, m$ positive integers.
Let $\beta : Z \to X$ be a finite cover from a normal $Z$ such that $f^{1/m} =: g \in \O_Z$ (for example, extract an $m^{\rm th}$ root of $f$ and normalize).
Since $\lambda_j$ is a jumping number of $\mJ(\omega_X, f^{\lambda})$, by Lemma \ref{lem.PropertiesOfMultiplier} (iii), $a_1$ is a jumping number of $\mJ(\omega_Z, g^{t})$.  Let $a_0$ denote the jumping number of $\mJ(\omega_Z, g^{t})$ immediately before $a_1$.
Note that $\lambda_{j} > a_0/m \geq \lambda_{j-1}$.

By choosing a small $U$, we can assume
that the following holds for all closed points $s \in U$:
\begin{itemize}
\item[(a)] $\beta_s:Z_s \to X_s$ is a finite separable cover with $Z_s$ normal.
\item[(b)] $\tau(\omega_{X_s}, f_s^{\lambda}) \subseteq \mJ(\omega_X, f^{\lambda})_s$ for all $\lambda_{j-1} \le \lambda <\lambda_j$.
\item[(c)] $\tau(\omega_{Z_s}, g_s^{a_0})=\mJ(\omega_{Z}, g^{a_0})_s$.
\end{itemize}
Applying Lemmas \ref{lem.PropertiesOfMultiplier} (iii) and \ref{lem.PropertiesOfTest} (iii) to (c), we have $\mJ(\omega_{X}, f^{a_0/m})_s = \tau(\omega_{X_s}, f_s^{a_0/m})$.
Since $\mJ(\omega_{X}, f^{\lambda_{j-1}}) = \mJ(\omega_{X}, f^{\lambda})$ for all $\lambda_{j-1} \leq \lambda \leq a_0/m$, we see immediately from (b) that $\mJ(\omega_{X}, f^{\lambda})_s = \tau(\omega_{X_s}, f_s^{\lambda})$ for $\lambda_{j-1} \leq \lambda \leq a_0/m$.

By Lemma \autoref{lem.KeyCase}, possibly shrinking $U$, we can prove that
$\tau(\omega_{Z_s}, g_s^{t})=\mathcal{J}(\omega_Z, g^{t})_s$ for all $a_0 \leq t < a_1$.
Applying Lemmas \ref{lem.PropertiesOfMultiplier} (iii) and \ref{lem.PropertiesOfTest} (iii) again, we have $\mJ(\omega_{X}, f^{\lambda})_s = \tau(\omega_{X_s}, f_s^{\lambda})$ for $a_0/m \leq \lambda \leq a_1/m=\lambda_j$.
The proof of the claim is completed.
\end{proof}
Now we return to the proof of the theorem.  For each jumping number $\lambda_j$, assuming Conjecture \autoref{MS conj} recalling Remark \autoref{rem.FinitelyManySmoothIsOk}, we can prove the result for $\lambda \in [\lambda_{j-1}, \lambda_j]$ by finding the finitely many $X^{(i)}_j$ guaranteed by the claim.  Since there are only finitely many jumping numbers in $[0,1]$, the theorem follows.
\end{proof}

\begin{remark}
The referee asked the natural question of whether assuming Conjecture \autoref{MS conj} for $V$ of dimension $\leq n - 1$ is enough to imply Conjecture \autoref{test ideal conj} for $X$ of dimension $\leq n$ (this might be expected since one would expect the $F$-singularities of $X_s$ to be governed by the ordinarity of the exceptional divisors of a resolution).  This does not  seem to work at least with our proof.
The problem is that at the start of the proof of Lemma \autoref{lem.KeyCase} we argue as in \cite[Theorem 5.10]{MustataSrinivasOrdinary} where one needs to assume Conjecture \autoref{MS conj} also in dimension $n$ (because of an appeal to \cite[Corollary 5.7]{MustataSrinivasOrdinary}).  Perhaps there are ways around this issue however.
\end{remark}

We conclude with a natural question that was pointed out by the referee and others.

\begin{question}
Does Conjecture \autoref{MS conj} imply that a variant of the non-LC ideal \cite{AmbroQuasiLog,FujinoNonLCSheaves,FujinoSchwedeTakagiSupplements} reduce to $\sigma(X_s, \Delta_s)$ for all $s$ in a Zariski-dense set of closed points $S \subseteq \Spec A$ (with notation as above)?
\end{question}

We hope to consider this question at a later time.

\def\cfudot#1{\ifmmode\setbox7\hbox{$\accent"5E#1$}\else
  \setbox7\hbox{\accent"5E#1}\penalty 10000\relax\fi\raise 1\ht7
  \hbox{\raise.1ex\hbox to 1\wd7{\hss.\hss}}\penalty 10000 \hskip-1\wd7\penalty
  10000\box7}
\providecommand{\bysame}{\leavevmode\hbox to3em{\hrulefill}\thinspace}
\providecommand{\MR}{\relax\ifhmode\unskip\space\fi MR}
\providecommand{\MRhref}[2]{%
  \href{http://www.ams.org/mathscinet-getitem?mr=#1}{#2}
}
\providecommand{\href}[2]{#2}


\begin{thebibliography}{CHSW08}

\bibitem[Amb03]{AmbroQuasiLog}
{\sc F.~Ambro}: \emph{Quasi-log varieties}, Tr. Mat. Inst. Steklova
  \textbf{240} (2003), no.~Biratsion. Geom. Linein. Sist. Konechno Porozhdennye
  Algebry, 220--239. {\sf\scriptsize MR1993751 (2004f:14027)}

\bibitem[Bei12]{BeilinsonpadicPeriods}
{\sc A.~Beilinson}: \emph{{$p$}-adic periods and derived de {R}ham cohomology},
  J. Amer. Math. Soc. \textbf{25} (2012), no.~3, 715--738. {\sf\scriptsize
  2904571}

\bibitem[BS15]{BhattScholzeWittAffGr}
{\sc B.~Bhatt and P.~Scholze}: \emph{The Witt vector affine Grassmannian},
preprint, 2015


\bibitem[BB11]{BlickleBockleCartierModulesFiniteness}
{\sc M.~Blickle and G.~B{\"o}ckle}: \emph{Cartier modules: finiteness results},
  J. Reine Angew. Math. \textbf{661} (2011), 85--123. {\sf\scriptsize 2863904}

\bibitem[BSTZ10]{BlickleSchwedeTakagiZhang}
{\sc M.~Blickle, K.~Schwede, S.~Takagi, and W.~Zhang}: \emph{Discreteness and
  rationality of {$F$}-jumping numbers on singular varieties}, Math. Ann.
  \textbf{347} (2010), no.~4, 917--949. {\sf\scriptsize 2658149}

\bibitem[BK86]{BlochKatoPAdicEtale}
{\sc S.~Bloch and K.~Kato}: \emph{{$p$}-adic \'etale cohomology}, Inst. Hautes
  \'Etudes Sci. Publ. Math. (1986), no.~63, 107--152. {\sf\scriptsize 849653
  (87k:14018)}

\bibitem[CHSW08]{CortinasHaesemayerSchlichtingWeibelCyclicHomology}
{\sc G.~Corti{\~n}as, C.~Haesemeyer, M.~Schlichting, and C.~Weibel}:
  \emph{Cyclic homology, cdh-cohomology and negative {$K$}-theory}, Ann. of
  Math. (2) \textbf{167} (2008), no.~2, 549--573. {\sf\scriptsize 2415380
  (2009c:19006)}

\bibitem[Del74]{DeligneHodgeIII}
{\sc P.~Deligne}: \emph{Th\'eorie de {H}odge. {III}}, Inst. Hautes \'Etudes
  Sci. Publ. Math. (1974), no.~44, 5--77. {\sf\scriptsize MR0498552 (58
  \#16653b)}

\bibitem[Doh08]{DohertySingularitiesOfGenericProjection}
{\sc D.~C. Doherty}: \emph{Singularities of generic projection hypersurfaces},
  Proc. Amer. Math. Soc. \textbf{136} (2008), no.~7, 2407--2415.
  {\sf\scriptsize 2390507 (2008k:14076)}

\bibitem[{Du~}81]{DuBoisMain}
{\sc P.~{Du~Bois}}: \emph{Complexe de de {R}ham filtr\'e d'une vari\'et\'e
  singuli\`ere}, Bull. Soc. Math. France \textbf{109} (1981), no.~1, 41--81.
  {\sf\scriptsize MR613848 (82j:14006)}

\bibitem[Esn90]{EsnaultHodgeTypeOfSubvarieties}
{\sc H.~Esnault}: \emph{Hodge type of subvarieties of {${\bf P}^n$} of small
  degrees}, Math. Ann. \textbf{288} (1990), no.~3, 549--551. {\sf\scriptsize
  1079878 (91m:14075)}

\bibitem[Fed83]{FedderFPureRat}
{\sc R.~Fedder}: \emph{{$F$}-purity and rational singularity}, Trans. Amer.
  Math. Soc. \textbf{278} (1983), no.~2, 461--480. {\sf\scriptsize MR701505
  (84h:13031)}

\bibitem[Fuj10]{FujinoNonLCSheaves}
{\sc O.~Fujino}: \emph{Theory of non-lc ideal sheaves: basic properties}, Kyoto
  J. Math. \textbf{50} (2010), no.~2, 225--245. {\sf\scriptsize 2666656
  (2011f:14031)}

\bibitem[FST11]{FujinoSchwedeTakagiSupplements}
{\sc O.~Fujino, K.~Schwede, and S.~Takagi}: \emph{Supplements to non-lc ideal
  sheaves}, Higher Dimensional Algebraic Geometry, RIMS K\^oky\^uroku Bessatsu,
  B24, Res. Inst. Math. Sci. (RIMS), Kyoto, 2011, pp.~1--47.

\bibitem[Gab04]{Gabber.tStruc}
{\sc O.~Gabber}: \emph{Notes on some {$t$}-structures}, Geometric aspects of
  Dwork theory. Vol. I, II, Walter de Gruyter GmbH \& Co. KG, Berlin, 2004,
  pp.~711--734.

\bibitem[Gab12]{GabberCommunicationWithBhatt}
{\sc O.~Gabber}: \emph{Private communication}, 2012.

\bibitem[GL01]{GoodwillieLichtenbaumCohomologicalBound}
{\sc T.~G. Goodwillie and S.~Lichtenbaum}: \emph{A cohomological bound for the
  {$h$}-topology}, Amer. J. Math. \textbf{123} (2001), no.~3, 425--443.
  {\sf\scriptsize 1833147 (2002h:14029)}

\bibitem[GR70]{GRVanishing}
{\sc H.~Grauert and O.~Riemenschneider}: \emph{Verschwindungss\"atze f\"ur
  analytische {K}ohomologiegruppen auf komplexen {R}\"aumen}, Invent. Math.
  \textbf{11} (1970), 263--292. {\sf\scriptsize MR0302938 (46 \#2081)}

\bibitem[Gro61]{EGAIII1}
{\sc A.~Grothendieck}: \emph{\'{E}l\'ements de g\'eom\'etrie alg\'ebrique.
  {III}. \'{E}tude cohomologique des faisceaux coh\'erents. {I}}, Inst. Hautes
  \'Etudes Sci. Publ. Math. (1961), no.~11, 167. {\sf\scriptsize MR0163910 (29
  \#1209)}

\bibitem[HY03]{HaraYoshidaGeneralizationOfTightClosure}
{\sc N.~Hara and K.-I. Yoshida}: \emph{A generalization of tight closure and
  multiplier ideals}, Trans. Amer. Math. Soc. \textbf{355} (2003), no.~8,
  3143--3174 (electronic). {\sf\scriptsize MR1974679 (2004i:13003)}

\bibitem[HS77]{HartshorneSpeiserLocalCohomologyInCharacteristicP}
{\sc R.~Hartshorne and R.~Speiser}: \emph{Local cohomological dimension in
  characteristic {$p$}}, Ann. of Math. (2) \textbf{105} (1977), no.~1, 45--79.
  {\sf\scriptsize MR0441962 (56 \#353)}

\bibitem[HH06]{HochsterHunekeTightClosureInEqualCharactersticZero}
{\sc M.~Hochster and C.~Huneke}: \emph{Tight closure in equal characteristic
  zero}, A preprint of a manuscript, 2006.

\bibitem[HJ13]{HuberJorder}
{\sc A.~Huber and C.~J\"{o}rder}: \emph{Differential forms in the h-topology},
  Available at {\tt{http://arxiv.org/abs/1305.7361}}.

\bibitem[JR03]{JoshiRajanOrdinarity}
{\sc K.~Joshi and C.~S. Rajan}: \emph{Frobenius splitting and ordinarity}, Int.
  Math. Res. Not. (2003), no.~2, 109--121. {\sf\scriptsize 1936581
  (2003i:14023)}

\bibitem[Kol12]{KollarQuotientsByFinite}
{\sc J.~Koll{\'a}r}: \emph{Quotients by finite equivalence relations}, Current
  developments in algebraic geometry, Math. Sci. Res. Inst. Publ., vol.~59,
  Cambridge Univ. Press, Cambridge, 2012, pp.~227--256, With an appendix by
  Claudiu Raicu. {\sf\scriptsize 2931872}

\bibitem[Kov99]{KovacsDuBoisLC1}
{\sc S.~J. Kov{\'a}cs}: \emph{Rational, log canonical, {D}u {B}ois
  singularities: on the conjectures of {K}oll\'ar and {S}teenbrink}, Compositio
  Math. \textbf{118} (1999), no.~2, 123--133. {\sf\scriptsize MR1713307
  (2001g:14022)}

\bibitem[Kov12]{KovacsTheIntuitiveDefinitionOfDuBois}
{\sc S.~J. Kov{\'a}cs}: \emph{The intuitive definition of {D}u {B}ois
  singularities}, Geometry and arithmetic, EMS Ser. Congr. Rep., Eur. Math.
  Soc., Z\"urich, 2012, pp.~257--266. {\sf\scriptsize 2987664}

\bibitem[KS11a]{KovacsSchwedeDBDeforms}
{\sc S.~J. Kov\'acs and K.~Schwede}: \emph{{Du Bois} singularities deform},
  arXiv:1107.2349, to appear in Advanced Studies in Pure Mathematics.

\bibitem[KS11b]{KovacsSchwedeDuBoisSurvey}
{\sc S.~J. Kov\'acs and K.~Schwede}: \emph{Hodge theory meets the minimal model
  program: a survey of log canonical and {D}u {B}ois singularities}, Topology
  of Stratified Spaces (G.~Friedman, E.~Hunsicker, A.~Libgober, and L.~Maxim,
  eds.), Math. Sci. Res. Inst. Publ., vol.~58, Cambridge Univ. Press,
  Cambridge, 2011, pp.~51--94.

\bibitem[Laz04]{LazarsfeldPositivity2}
{\sc R.~Lazarsfeld}: \emph{Positivity in algebraic geometry. {II}}, Ergebnisse
  der Mathematik und ihrer Grenzgebiete. 3. Folge. A Series of Modern Surveys
  in Mathematics [Results in Mathematics and Related Areas. 3rd Series. A
  Series of Modern Surveys in Mathematics], vol.~49, Springer-Verlag, Berlin,
  2004, Positivity for vector bundles, and multiplier ideals. {\sf\scriptsize
  MR2095472 (2005k:14001b)}

\bibitem[Lee07]{Lee}
{\sc B.~Lee}: \emph{Local acyclic fibrations and the de rham complex},
  Available at {\tt{http://arxiv.org/abs/0710.3147}}.

\bibitem[Ma13]{MaFInjectivityAndBuchsbaumSingularities}
{\sc L.~Ma}: \emph{{$F$}-injectivity and {B}uchsbaum singularities}, preprint,
  2013.

\bibitem[MVW06]{MazzaVoevodskyWeibel}
{\sc C.~Mazza, V.~Voevodsky, and C.~Weibel}: \emph{Lecture notes on motivic
  cohomology}, Clay Mathematics Monographs, vol.~2, American Mathematical
  Society, Providence, RI, 2006. {\sf\scriptsize 2242284 (2007e:14035)}

\bibitem[MV99]{MorelVoevodsky}
{\sc F.~Morel and V.~Voevodsky}: \emph{{${\bf A}^1$}-homotopy theory of
  schemes}, Inst. Hautes \'Etudes Sci. Publ. Math. (1999), no.~90, 45--143
  (2001). {\sf\scriptsize 1813224 (2002f:14029)}

\bibitem[Mus10]{MustataOrdinary2}
{\sc M.~Musta{\c{t}}{\v{a}}}: \emph{Ordinary varieties and the comparison
  between multiplier ideals and test ideals ii}, arXiv:1012.2915.

\bibitem[MS10]{MustataSrinivasOrdinary}
{\sc M.~Musta{\c{t}}{\v{a}} and V.~Srinivas}: \emph{Ordinary varieties and the
  comparison between multiplier ideals and test ideals}, arXiv:1012.2818.

\bibitem[Sch07]{SchwedeEasyCharacterization}
{\sc K.~Schwede}: \emph{A simple characterization of {D}u {B}ois
  singularities}, Compos. Math. \textbf{143} (2007), no.~4, 813--828.
  {\sf\scriptsize MR2339829}

\bibitem[Sch09]{SchwedeFInjectiveAreDuBois}
{\sc K.~Schwede}: \emph{{$F$}-injective singularities are {D}u {B}ois}, Amer.
  J. Math. \textbf{131} (2009), no.~2, 445--473. {\sf\scriptsize MR2503989}

\bibitem[ST08]{SchwedeTakagiRationalPairs}
{\sc K.~Schwede and S.~Takagi}: \emph{Rational singularities associated to
  pairs}, Michigan Math. J. \textbf{57} (2008), 625--658.

\bibitem[ST10]{SchwedeTuckerTestIdealFiniteMaps}
{\sc K.~Schwede and K.~Tucker}: \emph{On the behavior of test ideals under
  finite morphisms}, arXiv:1003.4333, to appear in J. Algebraic Geom.

\bibitem[ST12]{SchwedeTuckerTestIdealSurvey}
{\sc K.~Schwede and K.~Tucker}: \emph{A survey of test ideals}, Progress in
  Commutative Algebra 2. Closures, Finiteness and Factorization (C.~Francisco,
  L.~C. Klinger, S.~M. Sather-Wagstaff, and J.~C. Vassilev, eds.), Walter de
  Gruyter GmbH \& Co. KG, Berlin, 2012, pp.~39--99.

\bibitem[Smi97]{SmithFRatImpliesRat}
{\sc K.~E. Smith}: \emph{{$F$}-rational rings have rational singularities},
  Amer. J. Math. \textbf{119} (1997), no.~1, 159--180. {\sf\scriptsize
  MR1428062 (97k:13004)}

\bibitem[{Sta}14]{stacks-project}
{\sc T.~{Stacks Project Authors}}: \emph{{\itshape Stacks Project}}, available
  at {\tt{http://stacks.math.columbia.edu}}, 2014.

\bibitem[Ste81]{SteenBrinkDuBoisReview}
{\sc J.~H.~M. Steenbrink}: \emph{{MR}0613848 (82j:14006)}, Review of ``Complexe
  de de {R}ham filtr\'e d'une vari\'et\'e singuli\`ere'' for the AMS
  mathematical reviews database, 1981.

\bibitem[Ste83]{SteenbrinkMixed}
{\sc J.~H.~M. Steenbrink}: \emph{Mixed {H}odge structures associated with
  isolated singularities}, Singularities, Part 2 (Arcata, Calif., 1981), Proc.
  Sympos. Pure Math., vol.~40, Amer. Math. Soc., Providence, RI, 1983,
  pp.~513--536. {\sf\scriptsize MR713277 (85d:32044)}

\bibitem[SV96]{SuslinVoevodsky}
{\sc A.~Suslin and V.~Voevodsky}: \emph{Singular homology of abstract algebraic
  varieties}, Invent. Math. \textbf{123} (1996), no.~1, 61--94. {\sf\scriptsize
  1376246 (97e:14030)}

\bibitem[Tak04]{TakagiInterpretationOfMultiplierIdeals}
{\sc S.~Takagi}: \emph{An interpretation of multiplier ideals via tight
  closure}, J. Algebraic Geom. \textbf{13} (2004), no.~2, 393--415.
  {\sf\scriptsize MR2047704 (2005c:13002)}

\bibitem[Voe00]{VoevodskyCompletelyDecomposed}
{\sc V.~Voevodsky}: \emph{Homotopy theory of simplicial sheaves in completely
  decomposable topologies}, Available at
  {\tt{http://www.math.uiuc.edu/K-theory/443/cdstructures.pdf}}.

\bibitem[Yan85]{YanagiharaWeaklyNormal}
{\sc H.~Yanagihara}: \emph{On an intrinsic definition of weakly normal rings},
  Kobe J. Math. \textbf{2} (1985), no.~1, 89--98. {\sf\scriptsize MR811809
  (87d:13007)}

\end{thebibliography}
\end{document}